\documentclass{amsart}

\usepackage{amsmath}
\usepackage{amsxtra}
\usepackage{amscd}
\usepackage{amsthm}
\usepackage{amsfonts}
\usepackage{amssymb}
\usepackage{eucal}
\usepackage[all]{xy}
\usepackage{graphicx}
\usepackage{tikz-cd}
\usepackage{mathrsfs}
\usepackage{euler}
\usepackage{hyperref}
\usepackage{color}
\usepackage{longtable}
\usepackage{float}
\usepackage{caption}

\usepackage[colorinlistoftodos, textsize=tiny]{todonotes}
\def\listtodoname{List of Todos}
\def\listoftodos{\@starttoc{tdo}\listtodoname}

\RequirePackage{color}
\definecolor{myred}{rgb}{0.75,0,0}
\definecolor{mygreen}{rgb}{0,0.5,0}
\definecolor{myblue}{rgb}{0,0,0.65}

\usepackage{tikz}
\usepackage{tikz-cd}
\usetikzlibrary{matrix,arrows,decorations.pathmorphing}

\theoremstyle{plain}
  \newtheorem{thm}{Theorem}[section]
  \newtheorem{prop}[thm]{Proposition}
  \newtheorem{lem}[thm]{Lemma}
  \newtheorem{cor}[thm]{Corollary}
	\newtheorem{question}[thm]{Question}
	
\theoremstyle{definition}
  \newtheorem{defn}[thm]{Definition}
  \newtheorem{example}[thm]{Example}
	
\theoremstyle{remark}
	\newtheorem{rem}[thm]{Remark}
	
\numberwithin{equation}{section}

\newcommand\ssec{\subsection}
\newcommand\sssec{\subsubsection}
\newcommand\BH{{\mathbb H}}
\newcommand\BN{{\mathbb N}}
\newcommand\BC{{\mathbb C}}
\newcommand\BR{{\mathbb R}}
\newcommand\BQ{{\mathbb Q}}
\newcommand\BP{{\mathbb P}}
\newcommand\BZ{{\mathbb Z}}
\newcommand\Bk{{\Bbbk}}

\newcommand\sco{{\mathscr O}}

\newcommand\im{\text{im }}
\newcommand\proj{\text{Proj }}

\DeclareMathOperator{\ord}{ord}

\DeclareMathOperator\di{Div}
\newcommand\sx{\mathscr X}
\newcommand \subhalf[1]{\frac{{#1} - 1}{2{#1}}}
\newcommand{\halfcan}{L}
\DeclareMathOperator{\Supp}{Supp}
\DeclareMathOperator{\initial}{in_\prec}
\DeclareMathOperator{\Eff}{Eff}
\DeclareMathOperator{\sat}{sat}
\DeclareMathOperator{\newspan}{span}
\newcommand{\PSL}{\operatorname{PSL}}

\makeatletter
\newcommand{\customlabel}[2]{%
   \protected@write \@auxout {}{\string \newlabel {#1}{{#2}{\thepage}{#2}{#1}{}} }%
   \hypertarget{#1}{#2}
}
\makeatother

\title{Spin canonical rings of log stacky curves}

\author{Aaron Landesman}
\address[Aaron Landesman]{Department of Mathematics, Harvard University}
\email{aaronlandesman@college.harvard.edu}

\author{Peter Ruhm}
\address[Peter Ruhm]{Department of Mathematics, Stanford University}
\email{pruhm@stanford.edu}

\author{Robin Zhang}
\address[Robin Zhang]{Department of Mathematics, Stanford University}
\email{robinz16@stanford.edu}

\date{\today}

\begin{document}

\begin{abstract}
 	Consider modular forms arising from a finite-area quotient of the
	upper-half plane by a Fuchsian group. By the classical results of
	Kodaira--Spencer, this ring of modular forms may be viewed as the
	log spin canonical ring of a stacky curve. In this paper, we
	tightly bound the degrees of minimal generators and relations of
	log spin canonical rings. As a consequence, we obtain a tight bound
	on the degrees of minimal generators and relations for rings of
	modular forms of arbitrary integral weight.
\end{abstract}

\maketitle

%%%%%%%%%%%%%%%%%%%%%%%%%%%% Introduction %%%%%%%%%%%%%%%%%%%%%%%%%%%%%%%

\section{Introduction}
Let $\Gamma$ be a {\bf Fuchsian group}, i.e.~ a discrete subgroup of
$\PSL_2(\BR)$ acting on the upper half plane $\BH$ by fractional
linear transformations, such that $\Gamma \backslash \BH$ has
finite area.
We consider the graded {\bf ring of modular forms}
$M(\Gamma) = \bigoplus_{k = 0}^\infty M_k(\Gamma)$.
One of the best ways to describe the ring $M(\Gamma)$ is to write
down a presentation.
To do so, it is useful to have a bound on the
degrees in which the generators and relations can occur. In the
special case that $\Gamma$ has no odd weight modular forms, Voight
and Zureick-Brown give tight bounds \cite[Chapters 7-9]{vzb:stacky}.
The main theorem of this work extends their result to all Fuchsian
groups $\Gamma$.

We can now consider the orbifold $\Gamma \backslash \BH$ over $\BC$.
For example, in the case $\Gamma$ acts freely on $\BH$, $\Gamma
\backslash \BH$ is a Riemann surface over $\BC$. 
Although $\Gamma \backslash \BH$ may be non-compact, we can form a
compact Riemann surface $\Gamma \backslash \BH^*$ by adding in
cusps (with associated divisor of cusps $\Delta$).

In order to find generators and relations for $M(\Gamma)$, we
translate the seemingly analytic question of understanding the ring
of modular forms into the algebraic category, using a
generalization of the GAGA principle.
As shown by Voight and Zureick-Brown \cite[Proposition 6.1.5]
{vzb:stacky}, there is an equivalence of categories between orbifold
curves and log stacky curves over $\BC$. For the remainder
of the paper, we will work in the algebraic category.

Let $X$ be a smooth proper geometrically-connected algebraic curve of genus $g$ over a field $\Bk$.
It is well known that the canonical sheaf $\Omega _X,$ with
associated canonical divisor $K_X$, determines the {\bf canonical
map } $\pi: X \rightarrow \BP_\Bk^{g - 1}$. Then, the {\bf canonical
ring} is defined to be
\begin{align*}
	R(X, K_X) := \bigoplus_{d \geq 0} H^0(X, dK_X),
\end{align*}

\noindent
with multiplication structure corresponding to tensor product of
sections. In the case that $g \geq 2$, $\Omega_X$ is ample and
therefore $X \cong \proj R$. When $g \geq 2$, Petri's theorem 
shows that, in most cases, $R(X, K_X)$
is generated in degree 1 with relations in degree 2 (see
Saint-Donat \cite[p. 157]{saint-donat:proj} and Arbarello--Cornalba--Griffiths--Harris
\cite[Section 3.3]{acgh:algebraic-curves}). This has the pleasant 
geometric consequence that canonically embedded curves of genus $\geq 4$ which are not hyperelliptic curves, trigonal curves, or plane quintics are scheme-theoretically cut out by degree 2 equations.

Following Voight and Zureick-Brown \cite{vzb:stacky}, we generalize 
Petri's theorem in the direction of stacky curves equipped with
log spin canonical divisors. For a stacky curve $\sx$ with coarse space $X$ and
stacky points (also called ``fractional points'') $P_1, \ldots, P_r$
with stabilizer orders $e_1, \ldots, e_r \in \BZ_{\geq 2}$, we define

\[
	\di \sx = \left(\bigoplus_{P\notin \{P_1, \ldots, P_r\}} \langle 
	P \rangle \right) \oplus \left(\bigoplus_{i = 1}^r \left \langle 
	\frac{1}{e_i}P_i \right \rangle \right) \subseteq \BQ \otimes \di X.
\]
Then, a {\bf log spin curve} is a triple $(\sx, \Delta, \halfcan)$
where $\Delta \in \di X$ is a log divisor and $L \in \di \sx$ is a {\bf log spin canonical divisor}, meaning $2\halfcan \sim K_X +
\Delta + \sum_{i = 1}^{r} \frac{e_i-1}{e_i} P_i$. The central object of
study in this paper is the {\bf log spin canonical ring} of $(\sx,
\Delta, \halfcan),$ defined as
\begin{align*}
	R(\sx, \Delta, \halfcan) := \bigoplus_{k \geq 0} H^0(X, \lfloor k \halfcan \rfloor).
\end{align*}

\noindent
A brief overview of stacky curves, log divisors, and log spin
canonical rings is given in Subsection
~\ref{ssec:stacky-background}.

Our main theorem is to bound the degrees of generators and
relations of a log spin canonical ring. Let $\sx$ be a stacky curve
with signature $\sigma := (g; e_1, \ldots, e_r; \delta)$. The
application of O'Dorney's work \cite[Chapter 5]{dorney:canonical}
to log spin canonical rings gives a weak bound in the case $g = 0$
in terms of the least common multiples of the $e_i$'s. In their
treatment of log spin canonical rings, Voight and Zureick-Brown
\cite[Corollary 10.4.6]{vzb:stacky} bounded generator degrees by $6
\cdot \max(e_1, \ldots, e_r)$ and relation degrees by $12 \cdot \max
 (e_1, \ldots, e_ r)$ when $\halfcan$ is effective. Note that the
bounds we deduce differ from those stated in Voight and
Zureick-Brown \cite[Corollary 10.4.6]{vzb:stacky} by a factor of 2
because their grading convention differs from ours by a factor of 2.

These bounds are far from tight and do not collectively
cover all cases in all genera. The main theorem of this paper
gives significantly tighter bounds for the log spin canonical ring
of any log spin curve.

\begin{thm}
\label{thm:main}
Let $(\sx, \Delta, \halfcan)$ be a log spin curve over a perfect
field $\Bk$, so that $\sx$ has signature $\sigma = (g; e_1, \ldots,
e_r; \delta)$.

Then the log spin canonical ring is generated as a $\Bk$-algebra by 
elements of degree at most $e := \max(5, e_1, \ldots, e_r)$ with
relations generated in degrees at most $2e$,
so long as $\sigma$ does not lie in a finite list of exceptional
cases, as given in Table ~\ref{table:g-1-exceptional} for
signatures with $g = 1$ and Table ~\ref{table:g-0-exceptional} for
signatures with $g = 0$.
\end{thm}

\begin{rem}
\label{rem:main-weaker-assumptions}
In fact, the proof of Theorem ~\ref{thm:main} holds with $\Delta$ replaced by an arbitrary effective divisor of the coarse space.
 Furthermore, one may relax the assumption that $\Bk$ is perfect. Instead, one only need assume that the stacky curve
is separably rooted, as described further in Remark ~\ref{rem:stack-formalism}.
\end{rem}

Theorem ~\ref{thm:main} is proven separately in the cases that the genus $g = 0, g = 1,$ and $g \geq 2$ in Theorems ~\ref{thm:g-0-main}, ~\ref{thm:g-1-main}, and ~\ref{thm:g-high-main}, respectively.  In each of these proofs, we follow a similar inductive process utilizing the lemmas of Section ~\ref{sec:induction}; however, in the first two cases we explicitly construct specific base cases and present a finite list of exceptional cases, whereas in the genus $g \geq 2$ case we deduce base cases from more general arguments.

\begin{rem}
\label{rem:explicit-generators}
In addition to providing bounds on the degrees of generators and relations of log spin canonical rings, the proof
of the genus one and genus zero cases of our main theorem also yield explicit systems
of generators and initial ideals of relations, as described in Remarks ~\ref{rem:genus-1-explicit-bound} and ~\ref{rem:genus-0-explicit-bound}.
Furthermore, our proof of the genus $g \geq 2$ case provides
an inductive procedure for explicitly determining the generators
and initial ideal of relations of a log spin canonical ring given a presentation of the corresponding ring on the coarse space, but actually computing such a presentation of log spin canonical on the coarse space can be difficult.  Many
explicit systems of generators and relations for curves of genus $2
\leq g \leq 15$ are detailed in interesting examples by Neves
\cite[Section III.4]{neves:halfcan}. 
\end{rem}

\begin{rem}
The explicit construction described in Remark ~\ref{rem:explicit-generators} also reveals that the bounds given in Theorem ~\ref{thm:main} are tight.
In almost all cases, the log spin canonical ring requires a generator in degree $e = \max(5, e_1, \ldots, e_r)$ and a relation in degree at least $2e-4$. 
Furthermore, there are many infinite families of cases which require a generator in degree $e = \max(5, e_1, \ldots, e_r)$ and a relation in degree exactly $2e$. For further detail, see Remarks ~\ref{rem:genus-0-explicit-bound}, ~\ref{rem:genus-1-explicit-bound}, and ~\ref{rem:genus-2-explicit-bound} in the cases that the genus is $0,1,$ or $\geq 2$ respectively.
\end{rem}

Combining the main theorem of this paper, Theorem ~\ref{thm:main} with the main theorem from Voight and Zureick-Brown ~\cite
[Theorem 1.4]{vzb:stacky} and a minor Lemma \cite[Lemma 10.2.1]{vzb:stacky}
we have the following application to
rings of modular forms.

\begin{cor}
\label{cor:main-mod-forms}
Let $\Gamma$ be a Fuchsian group and $\sx$ the stacky curve
associated to $\Gamma \backslash \BH$ with signature $\sigma
= (g; e_1, \ldots, e_r; \delta)$. 

If $M_k(\Gamma) = 0$ for all odd $k$, then the ring of modular
forms $M(\Gamma)$ is generated as a $\BC$-algebra by elements of
degree at most $6 \cdot \max(3, e_1, \ldots, e_r)$ with relations
generated in degrees at most $12 \cdot \max(3, e_1, \ldots, e_r)$.

If there is some odd $k$ for which $M_k(\Gamma) \neq 0$, then the
ring of modular forms $M(\Gamma)$ is generated as a $\BC$-algebra
by elements of degree at most $\max(5, e_1, \ldots, e_r)$ with
relations generated in degree at most $2 \cdot \max(5, e_1, \ldots, 
e_r)$ so long as $\sigma$ does not lie in a finite list of
exceptional cases which are listed and described in Table
~\ref{table:g-1-exceptional} for signatures with $g = 1$ and in
Table ~\ref{table:g-0-exceptional} for signatures with $g = 0$.
\end{cor}

\begin{rem}
\label{rem:gen-at-most-three}
If $M(\Gamma)$ has some odd weight modular form, then it has an odd weight modular form in weight $3$. When $g \geq 2$, we see that this is true because $\dim_\Bk H^0(\sx, 3\halfcan) > 0$ by Riemann--Roch and the fact that $\deg \lfloor 3\halfcan \rfloor > 2g - 1$. When the genus is zero or one, we see that there is a generator in weight $1$ or weight $3$ in the base cases given in Table ~\ref{table:g-1-base} and Table ~\ref{table:g-0-base-cases}. Hence, there is an odd weight modular form in weight $3$ in general. A consequence of this observation is that the bound on the degree of generators and relations when $M(\Gamma)$ has some odd weight modular form, as given in Corollary ~\ref{cor:main-mod-forms}, is closely related to the degree of the minimal odd weight modular form.
\end{rem}

\begin{example}
\label{eg:congruence-bounds}
In this example, we deduce bounds on the weight of generators and relations of the ring of modular forms associated to any congruence subgroup $\Gamma \subset SL_2(\BZ)$. Since
the action  of $SL_2(\BZ)$ on $\BH$ only has points with stabilizer
order 1, 2 and 3, and has at least one cusp, the
action of $\Gamma$ on $\BH$ can only have points with stabilizer order 1, 2 and 3, and has at least one cusp.

If $\Gamma$ has no nonzero odd weight modular forms, then $\Gamma$ is generated in weight at most $6$ with relations in weight at most $12$. This follows from work by Voight and Zureick-Brown \cite[Theorem 1.4 and Theorem 9.3.1]{vzb:stacky}. Note that the exceptional cases of their result \cite[Theorem 9.3.1]{vzb:stacky}, which happen when the genus is zero, do not occur because $\delta > 0$.

If $\Gamma$ has some nonzero odd weight modular form, then it must 
have no points with stabilizer order 2 by Remark ~\ref{rem:odd-denom}. 
Therefore, by Corollary ~\ref{cor:main-mod-forms}, $M(\Gamma)$ is 
generated in weight at most $5$ with relations in weight at most $10$. Furthermore, it is not difficult to show that $M(\Gamma)$ is generated in weight at most 4 with relations in weight at most 8 when the genus of the stacky curve associated to $\Gamma \backslash \BH^*$ is 0 or 1, as noted in Remark ~\ref{rem:cusp-generation-4}. Note that the exceptional cases in Tables ~\ref{table:g-1-exceptional} and ~\ref{table:g-0-exceptional} do not occur because $\Gamma$ has a cusp, so $\delta > 0$.
\end{example}

\begin{rem}
\label{rem:factor-two}
In the case that $M_k(\Gamma) = 0$ for all odd $k$, the generation bound of $6 \cdot \max(3, e_1, \ldots, e_r)$ and relation
bound of $12 \cdot \max(3, e_1, \ldots, e_r)$ can be reduced to $2\cdot \max(3,e_1, \ldots, e_r)$ and $4 \cdot \max(3,e_1, \ldots, e_r)$, apart from several small families of cases. See 
~\cite[Theorem 9.3.1]{vzb:stacky} and ~\cite[Theorem 8.7.1]{vzb:stacky} for a more precise statement of these bounds in the cases that $g = 0$ and $g > 0$ respectively. Note that we multiply all bounds given in Voight and Zureick-Brown \cite{vzb:stacky} by a factor of two. Our grading convention for log spin canonical rings uses weight $k$ for the degree whereas Voight and Zureick-Brown $d = 2k$ for degree.
\end{rem}

The remainder of the paper will be primarily devoted to proving
Theorem ~\ref{thm:main}. The idea of the
proof will be to induct first on the number of stacky points and then on the stabilizer order of those points.  To this end, we first review important background in
Section ~\ref{sec:background}; providing essential examples in Section \ref{sec:examples}; develop various inductive
tools in Section ~\ref{sec:induction}; and prove Theorem ~\ref{thm:main} in genus $g\ge 2$, genus $g=1$, and genus $g=0$ in Sections ~\ref{sec:g-high},
~\ref{sec:g-1}, and ~\ref{sec:g-0} respectively.  
Finally, in Section
~\ref{sec:further-questions}, we pose several questions for future
research.

%%%%%%%%%%%%%%%%%%%%%%%%%%%% Background %%%%%%%%%%%%%%%%%%%%%%%%%%%%%%%

\section{Background}
\label{sec:background}
Here we collect various definitions and notation that will be
used throughout the paper. For basic references on the statements and definitions
used below, see Hartshorne \cite[Chapter IV]{hartshorne:ag},
Saint-Donat \cite{saint-donat:proj}, Arbarello--Cornalba--Griffiths--Harris 
\cite[Section III.2]
{acgh:algebraic-curves}, and Voight--Zureick-Brown \cite[Chapter 2,
Chapter 5]{vzb:stacky}.

For the remainder of this paper, fix an algebraically closed field $\Bk$.
This is no restriction on generality, as generator and relation degrees are preserved under base change to the algebraic closure.

\ssec{Stacky Curves and Log Spin Canonical Rings}
\label{ssec:stacky-background}
We begin by setting up the notation for stacky curves and canonical rings. Wherever possible, we opt for a more elementary scheme-theoretic approach, instead of a stack-theoretic one. See Remark ~\ref{rem:stack-formalism} for more details.

\begin{defn}
\label{defn:stacky-curve}
A \textbf{stacky curve} $\sx$ over an algebraically closed field $\Bk$ is the datum of a smooth proper integral scheme $X$ of dimension $1$, together with a finite number of closed points of $X$, $P_1, \ldots, P_r$, called {\bf stacky points}, with {\bf stabilizer orders} $e_1, \ldots, e_r \in \BZ_{\geq 2}.$ The scheme $X$ associated to a stacky curve $\sx$ is called the {\bf coarse space} of $\sx$.
\end{defn}

\begin{rem}
\label{rem:stack-formalism}
Stacky curves may be formally defined in the language of stacks, as
is done in the works of Voight and Zureick-Brown \cite{vzb:stacky},
Abramovich and Vistoli \cite{abramovich-vistoli:compactifying}, and
Behrend and Noohi \cite{behrend-noohi:uniformization}.

The results of this paper can be easily phrased in terms of the language of stacks. If one works over an arbitrary field $\Bk$ (which need not be algebraically closed) one can extend Theorem ~\ref{thm:main} to hold in the case that the stacky curve $\sx$ is tame and separably rooted, i.e.~ the residue field of each
of the stacky points is separable.

With this stack-theoretic description in mind, the remainder of this paper is primarily phrased using the language
of schemes.
\end{rem}

\begin{defn}
\label{defn:div-ex}
Let $\sx$ be a stacky curve over $\Bk$ with coarse space $X$ of genus $g$ and stacky points $P_1, \ldots,
P_r$ with stabilizer orders $e_1, \ldots, e_r \in \BZ_{\geq 2}$.
Then, we notate
\[
	\di \sx := \left(\bigoplus_{P\notin \{P_1, \ldots, P_r\}} \langle 
	P \rangle \right) \oplus \left(\bigoplus_{i = 1}^r \left \langle 
	\frac{1}{e_i}P_i \right \rangle \right) \subseteq \BQ \otimes \di X.
\]
\end{defn}

We can equip stacky curves with a log divisor $\Delta$ that is a 
sum of distinct points each with trivial stabilizer. 
A divisor
$\Delta$ of this form is called a \textbf{log divisor}. We use
$\delta := \deg \Delta$ to refer to the degree of the log divisor.
If $\sx$ has coarse space $X$ of genus $g$, then we say $\sx$ has
\textbf{signature} $\sigma = (g; e_1, \ldots, e_r; \delta)$.

\begin{defn}
\label{defn:divisor-floor}
If divisor $D \in \di \sx$ and $D = \sum_{i = 1}^{n} \alpha_i P_i$
with $\alpha_i \in \BQ$, the floor of a divisor $\lfloor D
\rfloor$ is defined to be $\lfloor D \rfloor := \sum_{i = 1}^{n}
\lfloor \alpha_i \rfloor P_i$.
\end{defn}

A pair of a stacky curve and a log divisor $(\sx, \Delta)$ is
called a \textbf{log stacky curve} and the study of their
canonical rings is the main focus of the work by Voight and 
Zureick-Brown \cite{vzb:stacky}. For this paper, we consider \textbf{log
spin curves} which are triples $(\sx, \Delta, \halfcan)$ where $\sx$
is a stacky curve, $\Delta$ is a log divisor, and $\halfcan \in \di
\sx$ satisfies $2 \halfcan \sim K_X + \Delta + \sum_{i = 1}^{r}
\frac{e_i - 1}{e_i} P_i$. Such a divisor $\halfcan$ is called a
\textbf{log spin canonical divisor} on $(\sx, \Delta)$. Throughout the paper, we use
the notation $\halfcan_X := \lfloor L \rfloor.$ to refer to the log spin
canonical divisor (also known as the half-canonical divisor, 
semi-canonical divisor, or theta characteristic) associated to the
coarse space $X$ of $\sx.$  (i.e.~ $\halfcan_X$ is a divisor such
that $2\halfcan_X \sim K_X + \Delta$). 
We define $H^0$ of a stacky divisor as follows.

\begin{defn}
\label{defn:h0-stacky}
Recall the standard notation for the line bundle $\sco(D)$ on an integral normal scheme $X$ associated to a divisor $D \in \di X$:
\[
	\Gamma(U, \sco(D)) := \{f \in \Bk(X)^\times : \di|_U f + D|_U \geq 0 \} \cup \{0\}.
\]

\noindent
Let $\sx$ be a stacky curve with coarse space $X$.
If $D \in \di \sx$ is a Weil divisor, then we define
\begin{align*}
	H^0(\sx, D) &:= H^0(X, \lfloor D \rfloor)\\
	H^0(\sx, \sco(D)) &:= H^0(\sx, D)\\
	h^0(\sx, \sco(D)) &:= \dim_\Bk H^0(\sx, \sco(D))
\end{align*}
If $R$ is a graded ring, then we let $(R)_k$ refer to the $k^\text{th}$
graded component of $R$.
\end{defn}

\begin{rem}
\label{rem:gorenstein}
	The log canonical ring, defined to be the direct sum of the even
	graded pieces of the log spin canonical ring, is Gorenstein. It is
	Cohen--Macaulay from \cite[Example 2.5(a)]
	{watanabe:demazure-normal-graded} and then Gorenstein by
	\cite[Corollary 2.9]{watanabe:demazure-normal-graded}. In
	particular, this tells us that a log spin canonical curve, the
	projectivization of a log spin canonical ring, is projectively
	Gorenstein.
\end{rem}

\begin{rem}
Although Definition ~\ref{defn:h0-stacky} may seem fairly ad hoc, it is naturally motivated in the context of stacks. See Voight and Zureick-Brown \cite[Lemma 5.4.7]{vzb:stacky} for a proof that Definition ~\ref{defn:h0-stacky} is equivalent to the stack-theoretic description.
\end{rem}

\begin{defn}
\label{defn:order-sup}
Let $D \in \di \sx.$ If $z \neq 0$ is a rational section of $\sco(D)$ denote the order of
zero of $z$ at $P$ by $\ord_P^D(z)$.
\end{defn}

\begin{defn}
\label{defn:log-spin-canonical-ring}
The {\bf log spin canonical ring} of $(\sx, \Delta, \halfcan)$ is
\[
	R(\sx, \Delta, \halfcan) := \bigoplus_{k \geq 0} H^0(\sx, k \halfcan).
\]
\end{defn}

\noindent
When the log spin curve is fixed, we usually use $R$ or $R_\halfcan$ to
represent $R(\sx, \Delta, \halfcan)$.

\begin{rem}
\label{rem:odd-denom}
Suppose $(\sx, \Delta, L)$ is a log spin curve. Note that $\halfcan$ is of the form
\[
	\halfcan = \sum_{i = 1}^{r} \subhalf{e_i} P_i + \sum_{i = 1}^{s} a_i Q_i
\]

\noindent
where $a_i \in \BZ$ and $e_i$ are odd. This is due to the fact
that $\halfcan \in \di \sx$: 
if some $e_i$ were even,
then $\frac{e_i - 1}{2e_i}$ would be in reduced form implying $\halfcan \notin \di \sx$.
\end{rem}

\begin{rem}
	Except in degenerate cases,
	such as when the signature is $(0;3,3,3;0)$
	as covered in the first line of Table ~\ref{table:g-0-exceptional},
	we have the following important restriction on the generators
	of $R_L$. For each $e_i$ in the signature of $\sx$, there
	will be at least one generator
	with degree $0 \bmod e_i$ and at least one generator
	with degree $-2 \bmod e_i$.
	Although this is an important restriction on the generators,
	we will not use this in the remainder of the paper.
\end{rem}

\begin{rem}
\label{rem:delta-not-1}
Suppose $(\sx, \Delta, \halfcan)$ is a log spin curve. Note that $\deg \Delta$
is even. because $2 \cdot \deg \halfcan = \deg \Delta + \deg \lfloor K
\rfloor = \deg \Delta + 2(g-1).$
In particular, we shall often use $\deg \Delta \neq 1$.
\end{rem}

\ssec{Saturation}
We define the notion of the saturation of a divisor, as can be
found in Voight and Zureick-Brown ~\cite[Section 7.2]{vzb:stacky}.
The classification of the saturations of log spin canonical
divisors are used in the proof of the main theorem and the various
lemmas in Section ~\ref{sec:induction}.

\begin{defn}
Let $D$ be a divisor on $\sx$. The \textbf{effective monoid} of $D$
is the monoid
\[
	\Eff(D) := \{k \in \BZ_{\geq 0} : \deg \lfloor kD \rfloor \geq 0 \}.
\]
\end{defn}

\begin{defn}
\label{defn:sat}
The \textbf{saturation of a monoid} $M \subseteq \BZ_{\geq 0}$,
denoted $\sat(M)$, is the smallest integer $s$ such that $M
\supseteq \BZ_{\geq s}$, if such an integer exists.
\end{defn}

\begin{rem}
For $D \in \di \sx$, we will often call $\sat(\Eff(D))$ the
\textbf{saturation of a divisor} $D$. For examples,
see Subsection ~\ref{ssec:g-0-saturation}.
\end{rem}

\ssec{Monomial Ordering}
\label{ssec:monomial-order}
Here we give a brief overview of the three monomial orderings
that we use. For further reference on monomial orderings,
initial ideals, and Gr\"{o}bner bases, see Eisenbud
\cite[Section 15.9]{eisenbud:comm-alg} and Cox--Little--O'Shea
\cite[Chapter 2]{cls:ideals-varieties-algorithms}.

\begin{defn}
\label{defn:monomial-degree}
Let $\Bk[x_1$, $\ldots$, $x_n]$ be a graded polynomial
ring with $\deg x_i = k_i$ and let $\alpha := \prod_{i = 1}^{n} x_{i}^{f_i} \in
\Bk[x_{1}$, $\ldots$, $x_{n}]$ be a monomial. Then
we define the {\bf degree} of $\alpha$ to be
\[
	\deg \alpha := \sum_{i = 1}^{n} k_i f_i.
\]
\end{defn}

\begin{defn}
\label{defn:grevlex}
The {\bf graded reverse lexicographic order}, or {\bf grevlex} $\prec_{\text{grevlex}}$ is defined as follows.
If $\alpha := \prod_{i = 1}^{n} x_{i}
^{f_i}$ and $\beta := \prod_{i = 1}^{n} x_{i}^{f_i'}$ are
monomials in $\Bk[x_{1}$, $\ldots$, $x_{n}]$, then $\alpha
\succ_{\text{grevlex}} \beta$ if either

\begin{equation}
	\deg \alpha = \sum_{i = 1}^{n} k_i f_i  > \sum_{i = 1}^{n} k_i f_i' = \deg \beta
\end{equation}
or
\begin{equation}
\label{eqn:grevlex-tie}
	\deg \alpha = \deg \beta \text{ and }	f_i < f_i' \text{ for
	the	largest } i \text{ such that } f_i \neq f_i'.
\end{equation}
\end{defn}

\begin{rem}
Note that the ordering of the variables matters in Equation
~\ref{eqn:grevlex-tie}.
\end{rem}

Our inductive arguments in Section ~\ref{sec:induction} will
usually have an inclusion $R \supseteq R'$ of log spin canonical
rings such that $R_\halfcan$ is generated by elements $x_{i}$ and $R$
is generated over $R_{\halfcan'}$ by elements $y_j$. In these cases, it is
natural to consider term orders which treat these sets of variables
separately.

\begin{defn}
\label{defn:block-order}
The \textbf{block term order} is defined as follows. Let $\Bk[y_{1}, \ldots, y_{m}]$ and $\Bk[x_{1}, \ldots,
x_{n}]$ be weighted polynomial rings with $\deg y_i = c_i, \: \deg x_i = k_i.$ Further assume we are given existing term
orders $\prec_y$ and $\prec_x$. Let $\alpha :=
\prod_{j = 1}^{m} y_{j}^{h_i}$ $\prod_{i = 1}^{n} x_{i}
^{f_i}$ and $\beta := \prod_{j = 1}^{m} y_{j}^{h_i'}
\prod_{i = 1}^{n} x_{i}^{f_i'}$ be monomials in $\Bk[y_{1},
\ldots, y_{m}] \otimes$ $\Bk[x_{1}, \ldots, x_{n}]$.
Let $\alpha_y := \prod_{j = 1}^{m} y_{j}^{h_i}$ be the
part of $\alpha$ in $\Bk[y_{1}, \ldots, y_{m}]$ and
likewise with $\alpha_x$, $\beta_y$, and $\beta_x$.

In the \textbf{(graded) block} (or \textbf{elimination}) term
ordering on $\Bk[y_{k_1', 1}, \ldots, y_{k_m', m}] \otimes \Bk[x_{k_1, 1},
\ldots, x_{k_n, n}]$, we define $\alpha \succ \beta$ if
  \begin{enumerate}
	\item[(i)] $\deg \alpha  > \deg \beta$ or
  \item[(ii)] $\deg \alpha = \deg \beta$ and $\alpha_y \succ_y \beta_y$ or
	\item[(iii)] $\deg \alpha = \deg \beta$ and $\alpha_y = \beta_y$ and $\alpha_x \succ_x \beta_x$.
  \end{enumerate}
\end{defn}

Now we give brief definitions of initial terms and Gr\"{o}bner
bases. These will be used in the proofs of the inductive lemmas
in Section ~\ref{sec:induction} as well as in the proof of Theorem
~\ref{thm:main}.

\begin{defn}
\label{defn:initial-term}
Let $\prec$ be an ordering on $\Bk[x_{1}$, $\ldots$, $x_{n}]$, 
with $\deg x_i = k_i$,
and let $f \in \Bk[x_{1}, \ldots, x_{n}]$ be a homogeneous
polynomial. The \textbf{initial term} $\initial(f)$ of $f$
is the largest monomial in the support of $f$ with respect to
the ordering $\prec$. Furthermore, we set $\initial(0) := 0$.
\end{defn}

\begin{defn}
\label{defn:initial-ideal}
Let $I$ be a homogeneous ideal of $\Bk[x_{1}$, $\ldots$,
$x_{n}]$. Then the {\bf initial ideal} $\initial(I)$ of $I$ is
the ideal generated by the initial terms of homogeneous polynomials
in $I$:
\[
	\initial(I) := \langle \initial(f) \rangle_{f \in I}
\]
\end{defn}

\begin{defn}
\label{defn:grobner-basis}
Let $I$ be a homogeneous ideal of $\Bk[x_{1}$, $\ldots$,
$x_{n}]$. A \textbf{Gr\"obner basis} for $I$, also known as a
\textbf{standard basis} for $I$, is a set of elements in $I$
such that their initial terms generate the initial
ideal of $I$.
\end{defn}

%%%%%%%%%%%%%%%%%%%%%%%%%%% Examples %%%%%%%%%%%%%%%%%%%%%%%%%%%%%%%

\section{Examples}
\label{sec:examples}

In this section, we work out several examples of computing presentations for spin canonical rings. In addition to providing intuition for the lemmas of Section ~\ref{sec:induction}, these examples also serve as useful base cases for our inductive proof of Theorem ~\ref{thm:main}.

\begin{example}
\label{eg:base-1-0}
Let $(\sx',0, \halfcan')$ be a log spin curve of genus $g = 1,$ with $\halfcan' = 0$.
Counting dimensions, we see $h^0(\sx, k\halfcan')=1$ for all $k \in \mathbb{Z}_{\ge 0}$ so it is immediately clear that $R_{\halfcan'} \cong \Bk[x]$ with $x$ a generator in degree 1.
\end{example}

\begin{rem}
	\label{remark:hilbert-series}
	In the following examples, in order to find the Hilbert series of
	a stacky curve, we will cite
	\cite[Theorem 4.2.1]{zhou:orbifold-riemann-roch-and-hilbert-series}.
	Note that
	\cite[Theorem 3.1]{buckleyRZ:ice-cream-orbifold-and-riemann-roch}
	restates \cite[Theorem 4.2.1]{zhou:orbifold-riemann-roch-and-hilbert-series}
	with the restriction that the dimension of the orbifold (which is the same
	as a stacky curve in dimension 1, by \cite[Proposition 6.1.5]{vzb:stacky},)	
	is strictly more than 1.
	The statement holds equally well when the dimension is $1$,
	but this restriction is included in
	\cite[Theorem 3.1]{buckleyRZ:ice-cream-orbifold-and-riemann-roch}
	because in birational geometry ``orbifolds''
	usually refer to a normal variety ramified only in codimension at least 2,
	while the stacky points we are dealing with appear in codimension 1.
\end{rem}

\begin{example}
\label{eg:adding-point}
Let $(\sx,3 \cdot \infty, L)$ be a stacky curve with coarse space $X$ and signature 
$(0; 3; 3)$. Let $P_1$ denote the lone stacky point which has stabilizer
order $3$ and suppose $\infty$ is a fixed closed point of $X$ that is
not equal to $P_1$.

Recall the notation $\halfcan_X = \lfloor L \rfloor \in \di X$
(i.e.~ the divisor without any stacky points). We will deduce the
structure of the log spin canonical ring $R_\halfcan$ from the
structure of the spin canonical ring $R_{\halfcan_X} := R(X, 3 \cdot
 \infty, \halfcan_X)$. This technique will later be generalized in
Lemma ~\ref{lem:sat-1}.

Note that $R_{\halfcan_X} \cong \Bk[x_1, x_2]$ where $\deg x_1 = \deg x_2 = 1.
$ To see this, observe that we will need two generators in degree $1$
because $h^0(X, \halfcan_X) = 2$ by Riemann--Roch. Because $\halfcan_X$ is very 
ample, we have that $R_{\halfcan_X}$ is generated in degree 1. To
conclude, note that $R_{\halfcan_X}$ does not have any relations. If there
exists some relation, then $\dim \proj R_{\halfcan_X} < 1$. This would
contradict the fact that $\halfcan_X$ is very ample. Thus, $\proj R_{\halfcan_X}
\cong X$ which has dimension 1.

Next, we construct generators and relations for $R_\halfcan$ using those
of $R_{\halfcan_X}$. Note that we have a natural inclusion $\iota: R_{\halfcan_X}
\hookrightarrow R_\halfcan$ induced by the inclusions $H^0(\sx, k\halfcan_X)
\hookrightarrow H^0(\sx, k\halfcan)$ for each $k \geq 0$. By Riemann--Roch,
we see there is some element $y_{1, 3} \in (R_\halfcan)_3$ with $\ord_{P_1}
(y_{1, 3}) = -1$, not in the image of the inclusion $\iota$. We
claim that there exist $a_1, a_2 \in \Bk$ and a degree $4$ polynomial
$f(x_1, x_2) \in \Bk[x_1, x_2]$ such that
\begin{align*}
	R_\halfcan \cong \Bk[x_1, x_2, y_{1, 3}] / (a_1 x_1 y_{1, 3} + a_2 x_2 y_{1, 3}
	+ f(x_1,x_2))
\end{align*}

First, note that $x_1, x_2, y_{1, 3}$ generate all of $R_\halfcan$ from the
Generalized Max Noether Theorem for genus zero curves from Voight
and Zureick-Brown \cite[Lemma 3.1.1]{vzb:stacky}. That is, the maps
\begin{align*}
	H^0(\sx, 3L) \otimes H^0(\sx, (k - 3)L) \rightarrow H^0(\sx, k\halfcan)
\end{align*}

\noindent
are surjective for $k \geq 4$. A relation of the form $a_1 x_1 y_{1,
3} + a_2 x_2 y_{1, 3} + f(x_1, x_2) = 0$ must exist because $h^0(\sx,
4L) - h^0(\sx, 4\halfcan_X) = 1,$ but $x_1 y_{1, 3}$ and $x_2 y_{1, 3}$
define two linearly independent elements with nontrivial image in
the 1-dimensional vector space $H^0(\sx, 4L) / H^0(\sx, 4\halfcan_X).$ So,
we obtain a surjection
\begin{align}
\label{align:surjection}
	\Bk[x_1, x_2, y_{1, 3}] / (a_1 x_1 y_{1, 3} + a_2 x_2 y_{1, 3} + f(x_1, x_2)) 
	\rightarrow R_\halfcan.
\end{align}

To complete the example, it suffices to show there are no
additional relations. 
One method would be to use \cite[Theorem 4.2.1]{zhou:orbifold-riemann-roch-and-hilbert-series}
to write down the Hilbert series and then check this agrees with the
Hilbert series of the ring we constructed above.
Here is an alternative method:
First, note that $a_1 x_1 y_{1, 3} + a_2 x_2
y_{1, 3} + f(x_1, x_2)$ is irreducible because there are no
relations among $x_1, x_2$ and $y_{1, 3}$ in lower degrees. Hence,
$\Bk[x_1, x_2, y_{1, 3}] / (a_1 x_1 y_{1, 3} + a_2 x_2 y_{1, 3} +
f(x_1, x_2))$ is integral and is 2-dimensional. Thus, the map
\eqref{align:surjection} defines a surjection from an integral
2-dimensional ring to a 2-dimensional ring. Therefore, it is an 
isomorphism.
\end{example}

\begin{example}
\label{eg:base-0-377}
Let $(\sx', 0, \halfcan')$ be a log spin curve with signature $\sigma =
(0; 3, 7, 7; 0)$ and $\halfcan' \sim -\infty + \frac{1}{3} P_1 +
\frac{3}{7} P_2 + \frac{3}{7} P_3,$ where $P_1, P_2,$ and $P_3$ are
distinct points. In this example, we will exhibit a minimal
presentation for $R' = R(\sx', 0, \halfcan')$ and show that $R'$ is
generated as a $\Bk$-algebra in degrees up to $e := \max(5, 3, 7, 7) = 7$
with relations generated in degrees up to $2e = 14$.
Notice that $\deg \lfloor k K_{\sx'} \rfloor= -k + \lfloor \frac{k}{3}
\rfloor + 2 \lfloor \frac{3k}{7} \rfloor$.

First, by \cite[Theorem 4.2.1]{zhou:orbifold-riemann-roch-and-hilbert-series},
the Hilbert series of this log spin curve is
\begin{align*}
	P_{(\sx', 0, \halfcan')}(k) &= \frac{1-2k + k^2 -2k^3 + k^4}{(1-k)^2} + \frac{k^3}{(1 - k)(1 - k^3)} + 2 \frac{k^3 + k^5 + k^7}{(1-k)(1-k^7)}.
\end{align*}
For $k = 0, 1, 2, \ldots$ we have
\[
	P_{(\sx', 0, \halfcan')}(k) = 1, 0, 0, 1, 0, 1, 1, 2, 1, 1, 2, 1, 3, 2, 3, \ldots
\]

\noindent
so $R'$ must have some generators $x_{3, 1}, x_{5, 1}, x_{7, 1},
x_{7, 2}$ with $x_{i, j} \in H^0(\sx', i \halfcan')$.

By the Generalized Max Noether Theorem for genus 0 curves (see Voight
and Zureick-Brown \cite[Lemma 3.1.1]{vzb:stacky}),
\begin{align}
\label{eqn:noether-377}
	H^0 (\sx', 21 \cdot \halfcan') \otimes H^0 (\sx',
	(k - 21) \halfcan') \rightarrow H^0 (\sx',
	k \halfcan')
\end{align}

\noindent
is surjective whenever $\deg (\lfloor (k - 21) \halfcan' \rfloor)
\geq 0$. It is fairly easy to see, by use of Riemann--Roch, that
the saturation of $\halfcan'$ is $5$ (see Definition ~\ref{defn:sat}).
Then the map in ~\eqref{eqn:noether-377} is surjective when $k \geq
21 + s = 26$ (i.e.~ $R'$ is generated up to degree 25).

To show that these generate all of $R'$, we need to show that all
$H^0 (\sx', k \halfcan')$ are generated by lower degrees for $k = 6$
and $7 < k \leq 25$. This can be seen by
checking these remaining cases via pole degree considerations
or using the generalized Max Noether's theorem. Thus, $R'$
is generated in degrees $\{3, 5, 7, 7\}$.

By relabelling the variables if necessary, we can assume that $x_{7
, 1}$ corresponds to the generator with maximal pole order at $P_2$ 
and $x_{7, 2}$ correspond to the generator with maximal pole order 
at $P_3$.
We then have two relations
\begin{align*}
	&a_1 x_{5, 1}^2 + a_2 x_{7, 2} x_{3, 1} + a_3 x_{7, 1} x_{3, 1} = 
	0 &\text{ in degree $10$} \\
	&b_1 x_{7, 2}^2 + b_2 x_{7, 2} x_{7, 1} + b_3 x_{7, 1}^2
	+ b_4 x_{5, 1} x_{3, 1}^3 = 0  &\text{ in degree $14$.}
\end{align*}

\noindent
Note that $a_1$ and $b_1$ are both nonzero. For example, if $a_1 = 0$,
we would have
\[
	a_2 x_{7, 2} x_{3, 1} = -a_3 x_{7, 1} x_{3, 1}
\]

\noindent
implying that $a_2 = a_3 = 0$, which would mean there is no relation at all.
A similar pole order consideration in forcing $b_1$ to be nonzero.

Let $I$ be the ideal generated by these relations in
$\Bk[x_{7, 2}, x_{7, 1}, x_{5, 1}, x_{3, 1}]$. Under grevlex with
$x_{3,1} \prec x_{5,1} \prec x_{7,1} \prec x_{7,2}$, the initial
ideal of $I$ is
\[
	\initial(I) = \langle x_{7, 2}^2, x_{5, 1}^2 \rangle
\]

\noindent
since $a_1$ and $b_1$ are nonzero. Since this ring has Hilbert series
equal to $P_{(\sx', 0, \halfcan')}(k)$, we have found all the relations.

Therefore, the canonical ring $R'$ has presentation $R' =
\Bk[x_{7, 2}, x_{7, 1}, x_{5, 1}, x_{3, 1}] / I$ with
initial ideal $\initial(I)$ generated by quadratics under grevlex
with $x_{3,1} \prec x_{5,1} \prec x_{7,1} \prec x_{7,2}$.
Thus, $R'$ is generated up to degree $e = 7$ with relations up to
degree $2e = 14$, as desired.
\end{example}

\begin{example}
\label{eg:base-1-33}
Let $(\sx',0 , \halfcan')$ be a log spin curve of genus $1$ with $\halfcan' = P - Q + \frac{1}{3}P_1 + \frac{1}{3}P_2$. In this example, we show that
$$R_{\halfcan'} \cong \Bk[u, x_3, y_3, y_4]/(x_3 y_3- \alpha uy_4, y_4^2 - \beta x_3^2 u - \gamma y_3^2u).$$

Let $u \in H^0(\sx,2\halfcan')$ be any nonzero element, let $x_3 \in H^0(\sx,3\halfcan')$ be an element with a pole at $P_1$ but not at $P_2$ and $y_3 \in H^0(\sx,3\halfcan')$ be an element with a pole at $P_2$ but not at $P_1$. Let $y_4 \in H^0(\sx,4\halfcan')$ be an element with a pole of order 1 at both $P_1$ and $P_2$. Note that $x_3$ and $y_3$ exist because the linear systems $3P - 3Q \sim P - Q$, $3P - 3Q + P_1$, and $3P - 3Q + P_1 + P_2$ are 0, 1, and 2 dimensional respectively.

Then, there exist constants $\alpha, \beta, \gamma \in \Bk$ so that 
$R_{\halfcan'} \cong \Bk[u, x, y_3, y_4]/(xy_3- \alpha uy_4, y_4^2 - \beta x^2 u - \gamma y^2u).$ The proof of this is fairly algorithmic: We may first 
write down the Hilbert series of $(R_{\halfcan'})_n$ over $\Bk$ using \cite[Theorem 4.2.1]{zhou:orbifold-riemann-roch-and-hilbert-series},
 then verify that these generators and relations produce the correct number of independent functions via an analysis of zero and pole order. The details are omitted as it is analogous to Example ~\ref{eg:base-0-377}.
\end{example}

\begin{example}
\label{eg:exception-1-5}
Let $(\sx', 0, L')$ be a log spin curve of genus 1 with $L' = P -
 Q + \frac{2}{5} P_1.$ Let $x_2 \in (R_\halfcan)_2$ be any nonzero element.
We obtain $\di x_2|_{P_1} = 0,$ since $2P - 2Q \sim 0$ and by
Riemann--Roch, $\dim_\Bk (P_L)_2 = 1.$ Let $y_3 \in (R_\halfcan)_3$
be any nonzero element. We obtain $\di y_3|_{P_1} = -P_1,$ by
Riemann--Roch, since if $y_3$ did not have a pole at $R$, we would
obtain $y_3 \in H^0 (\sx,3P-3Q) \cong H^0(\sx, P - Q) \cong 0$ as
$P \neq Q$. Finally, let $y_5 \in (R_\halfcan)_5$ be an element with
$y_5|_{P_1} = -P_1$. Then, we claim there is some $\alpha \in \Bk$
so that

$$R_\halfcan \cong \Bk[x_2 , y_3, y_5]/(y_3^4 - \alpha x_2 y_5^2).$$

In order to show this is an isomorphism, one can write down the Hilbert series using \cite[Theorem 4.2.1]{zhou:orbifold-riemann-roch-and-hilbert-series}
 and then use pole order considerations at $P_1$ to check the above relation exists. One can then check that the generators and relation determine a ring with the desired Hilbert series. The verification is analogous to Example ~\ref{eg:base-0-377} and is omitted.
\end{example}

\begin{rem}
\label{rem:base-0-377}
Examples ~\ref{eg:base-0-377}, \ref{eg:base-1-33}, and
\ref{eg:exception-1-5} are used as inductive base cases in the
genus 0 and genus 1 sections (see Tables
~\ref{table:g-0-base-cases} and ~\ref{table:g-1-base}).

We will come back to Examples ~\ref{eg:base-0-377} and
\ref{eg:base-1-33} in Examples ~\ref{eg:base-0-377-adm} and
\ref{eg:base-1-33-adm}, respectively, when checking an
admissibility condition that will be defined when introducing the
lemmas used in Subsection ~\ref{ssec:raise-orders} (see Definition
~\ref{defn:admissible}).
\end{rem}

%%%%%%%%%%%%%%%%%%%%%%%%%%% Inductive lemmas %%%%%%%%%%%%%%%%%%%%%%%%%%%%%%%

\section{Inductive Lemmas}
\label{sec:induction}
First we present several lemmas which provide the inductive steps
for the proof of the main theorem (Theorem ~\ref{thm:main}). In
Subsection ~\ref{ssec:add-points} we prove three lemmas which
determine the generators and relations of $R_\halfcan = R_{\halfcan'
+ \frac{\alpha }{\beta}P}$ from those of $R_{\halfcan'}$, where
$\halfcan' \in \BQ \otimes \di X$ and $\frac{\alpha}{\beta} \in \BQ$. 
In Subsection ~\ref{ssec:raise-orders}, we prove an inductive lemma
allowing us to transfer information about the log spin canonical 
ring of a stacky curve to those of stacky curves with stabilizer
orders incremented by $2$ and fixed log divisor and stacky points.

\ssec{Adding Points}
\label{ssec:add-points}
First, we give a criterion to determine if a set of monomials generates the initial ideal of relations of $\Bk[x_1, \ldots, x_m] \to R_D$.  This criterion will be used repeatedly to show that a given homogeneous ideal is in fact the ideal of relations.

\begin{lem}
\label{lem:relations_from_generators_induction} 
Suppose $\halfcan$, $\halfcan' \in \BQ$ $\otimes$ $\di X$ with $\halfcan=\halfcan'+\frac{\alpha}{\beta}P$, such that $R_{\halfcan'}$ generated by $x_1, \ldots, x_m$
and $R_{\halfcan}$ is minimally generated by $y_1, \ldots, y_n$ over 
$R_{\halfcan'}$.  Let $I'$ and $I$ be the ideals of relations of $\phi':\Bk[x_1, \ldots, x_m]\to R_{\halfcan'}$ and $\phi:\Bk[x_1, \ldots, x_m, y_1, \ldots y_n]\to R_{\halfcan}$ respectively.
Suppose there are sets of monomials $S\subseteq R_\halfcan-R_{\halfcan'}$ and $T\subseteq R_\halfcan-(S\cup R_\halfcan')$, and a monomial ordering $\prec$ such 
\begin{enumerate}
\item $S$ forms a $\Bk$-basis for $R_\halfcan$ over $R_\halfcan'$
\item $T \succ S\succ \Bk[x_1, \ldots, x_n]$ {\rm(}meaning all monomials in $T$ are bigger than all monomials in $S$ which are bigger than all monomials in $\Bk[x_1, \ldots, x_n]${\rm)}
\item  All monomials in $\Bk[x_1, \ldots, x_m, y_1, \ldots, y_n]$ lie in 
\begin{align*}
	S \; \cup\; \langle T\rangle \; \cup \; \initial(I') \Bk[x_1, \ldots, x_m, y_1, \ldots, y_n] \; \cup \; \Bk[x_1, \ldots, x_m]
\end{align*}
\end{enumerate}
Then,
\begin{align*}
	\initial(I) & = \initial(I') \Bk[x_1, \ldots, x_m, y_1, \ldots, y_n]
	+ \langle T \rangle.
\end{align*}
\end{lem}

\sssec*{Idea of Proof:} To show $\supseteq$, we show that $T \subseteq \initial(I)$, which follows immediately from $(1)$ and $(2)$.
We deduce $\subseteq$ by noting that we can reduce any monomial in $k[x_1, \ldots, x_m, y_1, \ldots, y_n]$ to a monomial in the basis $S$ via a set of relations whose initial terms include each monomial in $\initial(I')\Bk[x_1, \ldots, x_m, y_1, \ldots, y_n] + \langle T \rangle$.

\begin{proof}
First, notice that $I'\subseteq I$ so 
\[
	\initial(I)\supseteq \initial(I')\Bk[x_1, \ldots, x_m, y_1, \ldots, y_n].
\]

Now, let $f\in T$.  Since $S$ forms a $\Bk$-basis of $R_{\halfcan}$ over $R_{\halfcan'}$ by (1), we can write a relation $f - (\sum_{g\in S'} C_g g)-r=0$ for some finite subset $S'\subseteq S_{\deg(f)}$, $C_g\in \Bk$ for all $g\in S'$, and $r\in R_{\halfcan'}$.  This demonstrates that $\initial(I) \supseteq T$, and hence
\[
	\initial(I)\supseteq \initial(I')\Bk[x_1, \ldots, x_m, y_1, \ldots, y_n].
\]

To complete the proof, it suffices to show the reverse inclusion holds. By (2), any polynomial $G\in \Bk[x_1, \ldots, x_m, y_1, \ldots, y_n]$ with $\initial(\phi(G)) \in S$ cannot have a term in $T$.  Furthermore, since $S$ forms a $\Bk$-basis for $R_\halfcan$ over $R_{\halfcan'}$ by (1), and $\initial(\phi(G)) \in S,$ we obtain $\phi(G) \notin R_{\halfcan'} \subseteq R_\halfcan.$ Thus, $G=0$ is not a relation, so $f\not\in \initial(I)$.  Therefore, 
\[
	\initial(I)\subseteq R_\halfcan-S.
\]
In particular, there are no monomials in $I$ with initial terms in $S$.
Finally, note that
\[
	\initial(I) \cap \Bk[x_1, \ldots, x_m] = \initial(I').
\]

By (3), every monomial of $\Bk[x_1, \ldots, x_m, y_1, \ldots, y_n]$ is an element of either $S$, $\langle T\rangle$, or $\initial(I') \Bk[x_1, \ldots, x_m, y_1, \ldots, y_n].$ Therefore,
\begin{align*}
	\initial(I) & \subseteq \initial(I') \Bk[x_1, \ldots, x_m, y_1, \ldots, y_n] + \langle T \rangle.
\end{align*}
\end{proof}

To apply Lemma \ref{lem:relations_from_generators_induction}, we will need an appropriate monomial ordering.
The following definition provides the necessary ordering for the Lemma \ref{lem:sat-1}.
\begin{defn}
\label{defn:graded-p-lex}
Suppose $\halfcan$ is a divisor of $X$ such that $R_\halfcan$ is generated by $x_1, \ldots x_m$.  Then we have a map $\phi: \Bk[z_1, \ldots z_m] \rightarrow R_\halfcan, z_i\mapsto x_i$. If $P$ is a point in $X$, then $\phi$ defines a graded-$P$-lexicographic order (shortened to {\bf{graded $P$-lex}}) on $\Bk[z_1, \ldots, z_m]$ as follows.
If $f=\prod_{i=1}^m {z_i}^{q_i}$ and $g=\prod_{i=1}^m {z_i}^{r_i}$ with $f \neq g,$ then $f\prec g$ if one of the following holds:
\begin{enumerate}
\item $\deg(f) < \deg(g)$
\item $\deg(f) = \deg(g)$ and $-\ord_P(f) < -\ord_P(g)$
\item $\deg(f) = \deg(g)$, $-\ord_P(f) = -\ord_P(g)$, and $q_i > r_i$ for the largest $i$ such that $q_i\ne r_i$
\end{enumerate}
\end{defn}

\begin{rem}\label{rem:graded-P-lex-independent-of-line-bundle}
Observe that Definition \ref{defn:graded-p-lex} remains the same if we replace $-\ord_P$ with $-\ord_P^{L'}$ for any divisor $\halfcan'$ of $X$.
\end{rem}

\noindent One can easily verify graded $P$-lex is a monomial ordering in the sense defined in Cox--Little--O'Shea \cite[Chapter 2, $\mathsection$ 2, Definition 1]{cls:ideals-varieties-algorithms}.

We are almost ready to state Lemma ~\ref{lem:sat-1}, which will yield an inductive procedure for determining the generators and relations of $R_D$, where $D \in \di \BP^1$ is an effective $
\BQ$-divisor. Whereas O'Dorney considers arbitrary $\BQ$-divisors in
$Div(\mathbb{P}^1)$ \cite[Theorem 8]{dorney:canonical}, we restrict
attention to effective divisors and in Lemma ~\ref{lem:sat-1} we
obtain much tighter bounds.  Moreover, Lemma ~\ref{lem:sat-1}
also extends to curves of genus $g > 0$. 
We next prove the first of three lemmas used to inductively add points.

\begin{lem}
\label{lem:sat-1}
Let $X$ be a genus $g$ curve and let $\halfcan' \in \BQ \otimes \di X$
satisfy $h^0(X, \lfloor{\halfcan'}\rfloor)\ge 1$. Suppose $P$ is not a base-point of $k\halfcan'$ for all $k \in \BN$, meaning we can choose generators $u, x_1, \ldots, x_m$ of $R_{\halfcan'}$ in degree at most $\tau$ for some $\tau\in \BN$, with $\deg u = 1$, $\ord_P^{\halfcan'}(x_i)=0$ for all $1 \leq i \leq m$, and $\ord_P^{\halfcan'}(u) = 0$.  Suppose $\halfcan = \halfcan' + \frac{\alpha}{\beta} P$
for some $\alpha, \beta \in \BN$ such that $\frac{\alpha}{\beta}$ is reduced and
\begin{align}
\label{eqn:deg1-sat-ind}
	h^0(X, \lfloor k \halfcan \rfloor) &= h^0(X, \lfloor k \halfcan'
	\rfloor) + \left\lfloor k \frac{\alpha} {\beta} \right \rfloor &&\text{ for all } k \in \mathbb{
	N}.
\end{align}
Then,

\begin{enumerate}
\item[(a)] $R_{\halfcan}$ is generated over $R_{\halfcan'}$ by 
	elements $y_1, \ldots, y_n$ where $\deg(y_i) = k_i<\beta$, $-\ord_P
	^{L'}(y_i) = c_i$ for some $k_i$'s and $c_i$'s such that $c_i<c_{i+1}\le \alpha$ and $k_i\le k_{i+1}\le \beta$ for all $i$.

\item[(b)] Choose an ordering $\prec$ on $\Bk[u, x_1, \ldots, x_m]$ such that
	\[
		\ord_u(f) < \ord_u(h) \implies f\prec h.
	\]
	Equip $\Bk[y_1, \ldots, y_n]$ with graded $P$-lex, as defined in
	Definition ~\ref{defn:graded-p-lex}, and equip $\Bk[y_1, \ldots,
	y_n] \otimes \Bk[u, x_1, \ldots, x_m]$ with block order.
	If $I'$ is the ideal of relations of $\Bk[u, x_1, \ldots, x_m]
	\to R_{\halfcan'}$ and $I$ is the ideal of relations of
	$\Bk[u, x_1, \ldots, x_m, y_1, \ldots, y_n]	\to R_{\halfcan}$, then

	\begin{align*}
		\initial(I) &= \initial(I') \Bk[u, x_1, \ldots, x_m, y_1, \ldots, y_n] 
											 + \langle U_i: 1 \le i \le n-1 \rangle
											 + \langle V \rangle \\
	\end{align*}
	where $V = \{x_i y_j: 1 \le i \le m, 1 \le j \le n\}$ and $U_i$ is
	the set of monomials of the form $\prod_{j = 1}^{i} y_j^{a_j}$ with
	$a_j \in \BN_{\ge 0}$ such that
	\begin{enumerate}
		\item[\customlabel{custom:sat-1-*}{(U-1)}] $\sum_{j = 1}^i a_j c_j \ge c_{i+1}$, \\
		\item[\customlabel{custom:sat-1-**}{(U-2)}] there does not exist $(b_1, \ldots b_i) \ne (a_1,
			\ldots a_i)$ with all $b_j \le a_j$ and $\sum_{j = 1}^i b_j
			c_j \ge	c_{i + 1}$, \\
		\item[\customlabel{custom:sat-1-***}{(U-3)}] there does not
			exist $r<i$ such that $\sum_{j=1}^r a_j c_j> c_{r + 1}$.
	\end{enumerate}

\item[(c)] Let $\tau = \max(1, \max_{1 \leq i \leq m}(\deg(x_i)))$.
	Then, $R_\halfcan$ is generated over $R_\halfcan'$
	in degrees up to $\beta$ with $I$ generated over $I'$ in
	degrees up to $\max(2 \beta, \beta + \tau).$
\end{enumerate}
\end{lem}

\sssec*{Idea of proof:}
The proof will be fairly involved. To show part (a), we use Riemann-Roch to reduce the
problem to one of finding primitives of cones; we then apply previous work on continued fractions to deduce these primitives. To conclude this part of the problem, we show by dimension count that these primitives induce all the generators of $R_L$ over $R_L$'.

first define a set of generators of $R_\halfcan$ over $R_{\halfcan'}$. 
We then use Riemann--Roch
to count the dimension of $R_\halfcan$ over $R_{\halfcan'}$ and
show that the set of elements we produce forms a basis.

Next, part (b) immediately follows from the conclusion of Lemma
~\ref{lem:relations_from_generators_induction}, reducing the proof
to verifying the hypotheses of that lemma. The first two hypotheses
follow immediately from the definition of block order. Checking the
third condition is quite technical, but follows from the construction of $V$ and the $U_i$'s.

\begin{proof}
{\bf Part (a):}
By
Equation ~\ref{eqn:deg1-sat-ind}, for any $k
\in \BN$ such that $\lfloor k \frac{ \alpha}{\beta} \rfloor > 0$,

\[
	h^0 (X, k \halfcan ) = h^0(X, k \halfcan') + \left\lfloor k \frac{\alpha}{\beta} \right\rfloor.
\]

\noindent
Thus, there exist rational sections $t_i$ of $\sco(\lfloor k \halfcan \rfloor)$ with $\ord_P^{L'}(t_i) = i$
for any $i \in \{0, \ldots, \lfloor k \frac{\alpha}{
\beta} \rfloor \}$.
This reduces the problem at hand to finding the primitives of the cone in $\mathbb{Z}\times \mathbb{Z}$ with $x$ and $y$ coordinates,
bounded by the lines $y=0$ and $y=\frac{\alpha}{\beta}x$.

Cohn's geometric interpretation of Hirzebruch--Jung fractions, as described in 
\cite{popescuP:the-geometry-of-continued-fractions-and-the-topology-of-surface-singularities} 
yields an explicit formula for these primitives, as given in \cite[Proposition 4.3]{popescuP:the-geometry-of-continued-fractions-and-the-topology-of-surface-singularities}.
All primitives lie in degrees at most $\beta$
(since points in degree $\gamma=\beta+\omega$ can be written as a sum of points
in degree $\beta$ with those in degree $\omega$).  Note that if two primitives
had the same $y$ coordinate, then they would differ by a multiple of $(1,0)$
and thus could not both be generators. So, all the $c_i$'s (the $y$-coordinates
of the primitives) are distinct, and we may assume $c_i<c_{i+1}$ for
all $i$.  Furthermore, note that in any degree $k$ greater than 1, there can be
at most one primitive,
since we can construct points with $y$ coordinates
between 0 and $\lfloor k\frac{\alpha}{\beta}\rfloor$ as a sum of elements in
degree 1 and degree $k-1$, and these new primitives have pole order greater
than all primitives of smaller $x$-coordinate.  Therefore, the ordering of
primitives with $c_i<c_{i+1}$ for all $i$ also ensures that the $x$-values
$k_i$ are in weakly increasing order.  Finally, let $y_1, \ldots, y_n$
be elements of $R_L$
with pole orders given by $c_1,\ldots, c_n$ and degrees
given by the $x$ coordinates of the corresponding primitives.

We now recursively define a $\Bk$-basis for $R_\halfcan$ over $R_{\halfcan'}$. Define 
\[
	S_0 = \{u^l : l \in \BN_{\ge 0}\}
\]
and for each $i \in \BN$, suppose $c_j$ is the maximal element of $c_1, \ldots, c_n$ such that $c_j \le i.$ Note that such a $c_j$ exists because $1=c_1\le i.$ Define
\[
	S_i = y_j \cdot S_{i-j}.
\]
Since each $y_j$ has pole order $c_j$, this recursive construction ensures that 
\[
	z\in S_i \implies -\ord_P^{\halfcan'}(z)=i.
\]
Then define 
\begin{equation}\label{eqn:sat-1-defining-S}
	S = \bigcup_{i=1}^{\infty} S_i.  
\end{equation}
Note that $S_0$ is not part of this union and in fact $S\cap S_0 = \emptyset$ by pole order considerations.

By Equation ~\ref{eqn:deg1-sat-ind} for $k \in \BN$,
\[
	h^0(X, \lfloor k \halfcan \rfloor) = h^0(X, \lfloor k \halfcan'
		\rfloor) + \left\lfloor k \frac{\alpha} {\beta} \right \rfloor,
\] 
so $S$ contains elements in degree $k \in \BN$ with each pole order in $\{1 \ldots, \lfloor k\frac{\alpha}{\beta}\rfloor\}$.  Thus, by dimension counting, $S$ forms a $\Bk$-basis for $R_\halfcan$ over $R_{\halfcan'}$, and we have proven part (a).

{\bf Part (b):}
Let $S$ be as defined in Equation \ref{eqn:sat-1-defining-S}, define $U_i$ and $V$ as in the lemma's statement, and set
\[
	T = \left(\bigcup_{i=1}^{n-1} U_i \right) \cup V.
\]

We check that $S$, $T$, and $\prec$ meet the hypothesis of Lemma \ref{lem:relations_from_generators_induction}.  In part (a), we showed that $S$ forms a $\Bk$-basis for $R_\halfcan$ over $R_{\halfcan'}$ giving condition (1) of Lemma ~\ref{lem:relations_from_generators_induction}.  Our choice of monomial order in $\Bk[y_1, \ldots, y_n]$ and block order for $\Bk[y_1, \ldots, y_n]\otimes \Bk[u, x_1, \ldots, x_m, y_1, \ldots, y_n]$ implies that $T\succ S\succ \Bk[u, x_1, \ldots, x_m]$, giving condition (2) of Lemma \ref{lem:relations_from_generators_induction}.  

It only remains to check condition (3) of Lemma ~\ref{lem:relations_from_generators_induction}. To do this, suppose $f\in \Bk[u, x_1, \ldots, x_m, y_1, \ldots, y_n]$ is a monomial not contained in $\Bk[u, x_1, \ldots, x_m]$, meaning there is some $j$ such that $y_j|f$.  Further suppose $f\not\in \langle V \rangle$, meaning that for each $i\in \{1, \ldots m\}$, $x_iy_j \nmid f$. Since $y_j \mid f$ but $x_iy_j \nmid f$, we obtain $x_i\nmid f$.  Therefore, $f\in \Bk[u, y_1, \ldots, y_n]$.  We note that $S$ generates $y_j \cdot (\Bk[u,y_1, \ldots, y_n])$ as a $\Bk$-algebra.  That is, all monomials of $\Bk[u, x_1, \ldots, x_m, y_1, \ldots, y_n]$ are contained in
\[
	\Bk[u, x_1, \ldots, x_m] \cup V \cup \left(\bigcup_{i=1}^n y_i \cdot \Bk[u, y_1, \ldots, y_n]\right). 
\]
Notice that $S$ generates the ideal $(y_1, \ldots, y_n)$ considered as an ideal of the subring $\Bk[u, y_1, \ldots, y_n]$.  If $f\in S$, then $f=u^b \prod_{j=1}^n {y_j}^{a_j}$. Let $l$ be maximal such that $a_l\ne 0$.
Fix $i \in \{1, \ldots, n\}$. If $y_i \cdot f \notin S$, define
\[
b_j = \begin{cases}
	a_j &\text{ if } j \neq i\\
	a_j + 1 &\text{ if } j=i.
\end{cases}
\]

\noindent
Then, there is some $h \in \BN$ such that $i \le h \le \max(i, l)$ satisfying $\prod_{j = 1}^h y_j^{b_j}\not\in S\cup S_0$, 
and for all $r<h$ we have $\prod_{j=1}^r {y_j}^{b_j}\in S \cup S_0$.  
Choose some tuple $(\gamma_1, \ldots, \gamma_n)$ which is minimal, in the sense that we cannot decrease any $\gamma_j$ 
and have the following still satisfied: each $\gamma_j\le b_j$ 
and $\prod_{j=1}^h {y_j}^{\gamma_j}\not\in S\cup S_0.$ Our recursive definition of 
$S$ and the fact that $\prod_{j=1}^r {y_j}^{b_j}\in S\cup S_0$ implies that for each $1\le r< h$, we have $\prod_{j=1}^r {y_j}^{\gamma_j}\in S\cup S_0$.

We now check that $\prod_{j=1}^h y_j^{\gamma_j}\in U_{h}$, by checking conditions \ref{custom:sat-1-*}, \ref{custom:sat-1-**}, and \ref{custom:sat-1-***}.  
Notice that if $r\le n$, $\omega_1, \ldots \omega_r\in \mathbb{Z}_{\ge 0}$, and $\prod_{j=1}^r y_j^{\omega_j}\in S$, then our definition of $S$ implies $y_r\prod_{j=1}^r y_j^{\omega_j}\in S$ if and only if $c_r$ is maximal among $c_1, \ldots c_n$ not greater than than $c_r + \sum_{j=1}^r c_j \omega_j$.
Therefore, since $\prod_{j=1}^{h-1} y_j^{b_j}\in S$ but $\prod_{j=1}^h y_j^{b_j}\not\in S$, $c_h$ must not be maximal (among $c_1, \ldots, c_n$) such that $c_h \le \sum_{j=1}^h b_jc_j$, which means $c_{h+1}\le \sum_{j=1}^h b_jc_j$.  Therefore $\sum_{j=1}^h {y_j}^{\gamma_j}$ satisfies \ref{custom:sat-1-*}.  

Next, suppose we choose $\omega_1, \ldots, \omega_h$ such that
$\omega_j\le \gamma_j$ for all $j$ and $\omega_l \le \gamma_l$ for
some $l$.  Then, for all $r< h$ we have $\prod_{j=1}^r
{y_j}^{\gamma_j}\in S$, implying that for all $r< h$ we also have
$\prod_{j=1}^r {y_j}^{\omega_j}\in S$.  Furthermore, since $(\gamma_1,
\ldots, \gamma_h)$ was chosen to be minimal to satisfy the previous
condition and that $\prod_{j=1}^h {y_j}^{\gamma_j}\not\in S\cup S_0$, we have $\prod_{j=1}^h {y_j}^{\omega_j}\in S$.  Therefore, $c_h$ is minimal among $c_1, \ldots c_n$ that is not greater than $\sum_{j=1}^h c_j \omega_j$, so in particular $\sum_{j=1}^h c_j\omega_j < c_{h+1}$; therefore, $\prod_{j=1}^h {y_j}^{\gamma_j}$ satisfies condition \ref{custom:sat-1-**}.  
Since for each $r<h$, we have $\prod_{j=1}^r y_{j}^{\gamma_j}\in S$ meaning that $\sum_{j=1}^r \gamma_j c_j < c_{r+1}$, condition \ref{custom:sat-1-***} holds for $\prod_{j=1}^h y_j^{\gamma_j}$.  
Thus $\prod_{j=1}^h y_j^{\gamma_j}\in U_{h}$.

Since the ideal in $\Bk[u, y_1, \ldots, y_n]$ generated by is $S$ is $(y_1, \ldots, y_n) \cdot \Bk[u, y_1, \ldots, y_n]$ and $\bigcup_{i=1}^{n-1} U_i$ contains every monomial in $\bigcup_{i=1}^n y_i \cdot \Bk[u, y_1, \ldots, y_n]-S$, we have shown that all monomials of $\Bk[u, x_1, \ldots, x_m, y_1, \ldots, y_n]$ are contained in
\begin{align*}
				& \Bk[u, x_1, \ldots, x_m] \cup \langle V \rangle \cup S \cup \left\langle \bigcup_{i=1}^{n-1} U_i \right\rangle \\
	\subseteq \; 	& S \cup \langle T\rangle \cup \initial(I') \Bk[u, x_1, \ldots, x_m, y_1, \ldots, y_n] \; \cup \; \Bk[u, x_1, \ldots, x_m].
\end{align*}
This shows condition (3) of Lemma ~\ref{lem:relations_from_generators_induction} holds.  Thus, the conditions of 
Lemma ~\ref{lem:relations_from_generators_induction} are met. Finally, Lemma ~\ref{lem:relations_from_generators_induction} implies part (b).

{\bf Part (c):}
Finally, (c) of this lemma follows immediately by looking at the constructions of parts (a) and (b).  
\end{proof}

\begin{rem}\label{rem:quad-gen}
If $\frac{\alpha}{\beta}=\frac{e_i-1}{2 e_i}$ for some odd $e_i \in \BN_{\geq 3}$, then $T,$ as defined in the beginning of the proof of (b) in Lemma \ref{lem:sat-1} consists only of terms of the form $x_i y_j$ and $y_i y_j$, which are quadratic in the generators.
\end{rem}

\begin{rem}\label{rem:sat-1-gen-lem-generic}
The generators in Lemma \ref{lem:sat-1} are generic if $\frac{\alpha}{\beta}\le 1$ (since there is at most one positive best lower approximation $\frac{c_i}{k_i}$ with $k_i=1$).  When $\frac{\alpha}{\beta}>1$, the choice of generators in degrees great than 1 is generic; furthermore, we can make the choice in degree 1 generic by choosing $\lfloor \frac{\alpha}{\beta}\rfloor$ linearly independent elements in degree 1 with pole at $P$ of order $\lfloor \frac{\alpha}{\beta}\rfloor$ rather than elements with poles of order $1, \ldots, \lfloor \frac{\alpha}{\beta}\rfloor$; this requires minor complications in the construction of generators of the ideal of relations.
\end{rem}

We now restrict our attention to log canonical rings of stacky curves.  Lemma \ref{lem:sat-1} accounts for many of the induction cases when the spin canonical ring is saturated in degree 1, as defined in Definition ~\ref{defn:sat}.  We complement Lemma ~\ref{lem:sat-1} with the following two lemmas that allow us to inductively add points, under certain conditions when the spin canonical ring is saturated in degree two or three.

\begin{lem}
\label{lem:sat-2}
Let $(\sx, \Delta, \halfcan)$ and $(\sx', \Delta, \halfcan')$ be log spin curves with the same coarse space $X = X'$ having signatures $(g; e_1, \ldots, e_r; \delta)$ and $(g, e_1, \ldots, e_{r- 1}, \delta)$, where $e_r=3$.  Suppose $g > 0,$ and, if $g = 1,$ then $\deg 3\halfcan' \geq 2.$ Then, by Riemann--Roch $\sat(Eff(L'))\le 2$.
Furthermore, let $R_{\halfcan'} = \Bk[x_2, x_3
, x_5, \ldots, x_m]/I'$ and let $\halfcan = \halfcan' + \frac{1}{
3}P$, where $P\in X$ is a base point of $\halfcan'$ (which includes the case when $H^0(X, \lfloor \halfcan'\rfloor) = 0$).
Suppose for $i \in \{2, 3\}$, we generically choose $x_i$ satisfying $\deg x_i = i$ and $\ord_P^{\halfcan'}(x_i)= 0$. Choose an ordering on $\Bk[x_2, \ldots, x_m]$ that satisfies
\begin{align*}
	\ord_{x_2}(f) < \ord_{x_2}(h) \implies f \prec h.
\end{align*}

\noindent
Then, the following statements hold.

\begin{enumerate}
	\item[(a)] General elements  $y_i \in H^0(\sx,iL)$ for $i \in
		\{3, 4\}$ satisfy $-\ord_P^{\halfcan'}(y_i) = 1$ and any such choice of elements
		$y_3, y_4,$ minimally generate $R_\halfcan$ over $R_{\halfcan'}$.
	\item[(b)] Equip $\Bk[y_3, y_4]$ with grevlex so that $y_3 \prec 
		y_4$
		and equip the ring $\Bk[y_3, y_4] \otimes \Bk[x_2, \ldots, x_m]$ with block 
		order. Then,
		\begin{align*}
			\initial(I) &= \initial(I')\Bk[x, y_3, y_4]+\langle y_4 x_j \mid 2 \leq j \leq m \rangle +\langle y_4^2 \rangle.
		\end{align*}
\end{enumerate}
\end{lem}

\begin{proof}
First, note that when genus is at least we shall show the assumptions on $g$ imply $H^0(\sx, 3 \halfcan), H^0(\sx, 4 \halfcan)$ are both basepoint-free: If $g \geq 2$ then $\deg 3 \halfcan > 2g - 1$ 
and $\deg 4 \halfcan > 2g - 1$, so $H^0(\sx, 3 \halfcan)$ and $H^0(\sx, 4 \halfcan)$ are base point free. If $g = 1$, we assume $\deg 3 \halfcan \geq 2 > 2g - 1$, so we also have $\deg 4 \halfcan \geq 2 > 2g - 1,$ so again $H^0(\sx, 3 \halfcan)$ and $H^0(\sx, 4 \halfcan)$ are base point free. 

Therefore, general elements $y_3$ and $y_4$ satisfy $-\ord_P^{\halfcan'}(y_i) = 1$ by Riemann--Roch, proving part (a).

A quick computation checks that the set

\begin{align*}
	S =	& \; \{ y_3^ax_2^b x_3^\epsilon \mid a \geq 0, b 
	\geq 0, \epsilon \in \{0, 1\}\} \cup \{ y_3^ay_4 \mid a \geq 0 \}
\end{align*}

\noindent
is a $\Bk$ basis for $R_\halfcan$ over $R_{\halfcan'}$, thus
completing part (a).

Letting
\begin{align*}
	T =   &\; \{ y_4 x_j \mid 2 \leq j \leq m \}\cup \{ y_4^2 \}
\end{align*}
a similar (but much easier) computation to that of lemma $\ref{lem:sat-1}$ determines that $S$, $T$, and $\prec,$ using the ordering defined in (b), meet the conditions of 
Lemma \ref{lem:relations_from_generators_induction}. 
Hence, by Lemma \ref{lem:relations_from_generators_induction}, part (b) holds.
\end{proof}

\begin{lem}
\label{lem:sat-3}
Suppose $\halfcan'$ is a log spin canonical divisor of $\sx'$ with coarse
space $X$ of genus 0 such that $\sat(\Eff(\halfcan')) = 3$ and $R_
{\halfcan'} \cong \Bk[x_3, x_4 , x_5, \ldots, x_m]/I'$. Choose $x_3, \ldots, x_m$ such that $-\ord^{L'}_P(x_i)=0$ for all $i,$ which is possible as $X$ has genus $0$. Let $L = L' + \frac
{1}{3}P$. Suppose $\deg x_i = i$ for $i \in \{3, 4, 5\}$ and that
the ordering on $\Bk[x_3, \ldots, x_m]$ satisfies
\begin{align*}
	\ord_{x_3}(f) < \ord_{x_3}(h) \implies f \prec h.
\end{align*}

\noindent
Then, the following statements hold.

\begin{enumerate}
	\item[(a)] General elements  $y_i \in H^0(\sx, iL)$ for $i \in \{3,
		4,5\}$ satisfy $-\ord_P^{\halfcan'}(y_i) = 1$ and any such choice of elements $y
		_3, y_4,$ and $y_5$ minimally generate $R_\halfcan$ over $R_{\halfcan'}$.
	\item[(b)] Equip $\Bk[y_3, y_4, y_5]$ with grevlex so that $y_3 \prec 
		y_4 \prec y_5$
		and equip the ring $\Bk[y_3, y_4, y_5] \otimes \Bk[x_3, \ldots, x_m]$ with the block order.  Then,
		\begin{align*}
			\initial(I) &= \initial(I) \Bk[y_3, y_4, y_5, x_3, \ldots, x_m] \\
			&+ \langle y_i x_j \mid 4 \leq i \leq 5, 3 \leq j \leq m\rangle \\
			&+ \langle y_i y_k \mid 4 \leq i \leq j \leq 5\rangle.
		\end{align*}
\end{enumerate}
\end{lem}

\begin{proof}
Since $X$ has genus 0, general
elements $y_3, y_4,$ and $y_5$ in weights 3, 4, and 5 respectively satisfy $-\ord_P^{\halfcan'}(y_i) =
1$. We see by pole order considerations that

\begin{align}
\label{eqn:add_one_generator}
	\begin{split}
		S=	&\{ y_3^ax_3^b x_4^\epsilon x_5^{\epsilon'} \mid a \geq 0, b 
		\geq 0,(\epsilon, \epsilon') \in \{(0,0),(0,1),(1,0)\} \} \\
		\cup \;&\{y_3^ay_4, y_3^by_5 \mid a \geq 0, b \geq 0 \}
	\end{split}
\end{align}

\noindent forms a $\Bk$ basis for $R_\halfcan$ over $R_{\halfcan'}$, which concludes part (a) of the proof.

Setting
\[
	T = \{ y_i x_j \mid 4 \leq i \leq 5, 3 \leq j \leq m\} \cup \{ y_i y_k \mid 4 \leq i \leq j \leq 5\} 
\]
we can argue similarly to Lemma \ref{lem:sat-1} that $S$ and $T$ along with $\prec$ satisfy the hypothesis of Lemma \ref{lem:relations_from_generators_induction}, concluding part (b).
\end{proof}

One can prove similar results in cases with different conditions on saturation, base-point freeness, and the coefficients of added points, but only the cases of Lemmas ~\ref{lem:sat-1}, ~\ref{lem:sat-2}, and ~\ref{lem:sat-3} are needed for the remainder of this paper.  We next turn to an inductive method to increment the $e_i$'s.

\ssec{Raising Stabilizer Orders}
\label{ssec:raise-orders}
In this subsection, we present Lemma ~\ref{lem:raise-stacky-order},
whose proof is almost identical to one of Voight and Zureick-Brown
\cite[Theorem 8.5.7]{vzb:stacky}. Lemma
~\ref{lem:raise-stacky-order} implies that if the main result,
Theorem ~\ref{thm:main}, holds for a curve with signature $(g; e_1',
\ldots, e_\ell', e_{\ell + 1}' \ldots, e_r';\delta)$ with $e_{\ell + 1}' =
\cdots = e_r'$ satisfying an admissibility condition (cf. Definition
~\ref{defn:admissible}), then Theorem ~\ref{thm:main} also holds for
a curve with signature $(g; e_1', \ldots, e_\ell', e_{\ell + 1}' + 2,
\ldots, e_r' + 2; \delta)$.

First, we define a notion of admissibility that is quite similar to
the admissibility defined by Voight and Zureick-Brown
\cite[Definition 8.5.1]{vzb:stacky}. Our notion is an adaptation
the case of log spin canonical divisors.

One key difference between the notion of admissibility in Definition ~\ref{defn:admissible} and that of Voight and Zureick-Brown \cite[Definition 8.5.1]{vzb:stacky} is that
we cannot assume that $\{P_i\} \cap \Supp(\halfcan_X') = \emptyset$, as $ \halfcan_X'$ may have no nonzero global sections. We circumvent this issue by working with the orders of zeros and poles relative to $\halfcan_X'$, rather than relative to the $\sco_X,$ using Definition ~\ref{defn:order-sup}.

\begin{defn}
\label{defn:admissible}
Let $(\sx', \Delta, L')$ be a log spin
curve, with coarse space $X$ and stacky points $Q_1, \ldots, Q_r$.  Let $J \subset
\{1, \ldots, r\}$. Let $e_i := e_i'+ 2 \chi_J (i)$ where

\[
\chi_J(i) = \begin{cases}
	1, &\text{ if }i \in J\\
	0, &\text{ otherwise. } 
\end{cases}
\]

Let $R'$ be the canonical ring associated to $\sx'$. Define $(\sx', \halfcan', J)$
to be {\bf admissible} if $R'$ admits a presentation
\begin{align*}
	R' \cong \left( \Bk[x_1, \ldots, x_m] \otimes \Bk[y_{i, e_i'}]_{i \in J} \right)/I'
\end{align*}

\noindent
with each $y_{i, e_i'}$ viewed in $R'$ through the image of this isomorphism and such that for each $i \in J$ such that the following
three conditions hold:
\begin{enumerate}
	\item[\customlabel{custom:Ad-i}{(Ad-i)}] First, 
		\begin{align*}
		\deg y_{e_i'} &= e_i'  &\text{ and } &&-\ord_{Q_i}^{\halfcan_X'}(y_{i, e_i'})
			= \frac{e_i'- 1}{2}.
		\end{align*}
	\item[\customlabel{custom:Ad-ii}{(Ad-ii)}] Second, every generator $z 
		\in \{x_1, \ldots, x_m\} \cup \{y_{j, e_i'}: j\in J - \{i\} \}$ 
		satisfies
		\begin{align*}
			\frac{-\ord_{Q_i}^{\halfcan_X'}(z)}{\deg z} < \subhalf {e_i'}.
		\end{align*}
	\item[\customlabel{custom:Ad-iii}{(Ad-iii)}] Third, we have
		\begin{align*}
			\deg \lfloor e_i L' \rfloor \geq \max(2g - 1,0) + \max_{k \geq 0} \# S_{\sigma, J}(i, k)
		\end{align*}
		where
		\begin{align*}
			S_{\sigma, J}(i, k) := \{j \in J : j \neq i \text{ and } e_j'+2k
			\mid e_i - e_j'\}.
		\end{align*}
\end{enumerate}
\end{defn}

Before using this admissibility condition in the lemmas of
this section, we give a few explicit examples for which admissibility
holds.

\begin{example}
\label{eg:base-0-377-adm}
Here, we explicitly check admissibility in the context of
Example ~\ref{eg:base-0-377}. Recall that the setup is that
$(\sx', 0, \halfcan')$ is a log spin curve with signature $\sigma := (0; 3, 7, 7; 0)$
and $\halfcan' \sim -\infty + \frac{1}{3}P_1 + \frac{3}{7}P_2 + \frac{3}{7}P_3$,
where $P_1, P_2$, and $P_3$ are distinct points. We demonstrate
that $(\sx', \halfcan', \{2, 3\})$ is admissible.

As shown in Example ~\ref{eg:base-0-377}, the log spin canonical
ring corresponding to $(\sx', 0, \halfcan')$ has a presentation
$\Bk[x_{7, 2}, x_{7, 1}, x_{5, 1}, x_{3, 1}] / I$ with
$x_{i, j} \in H^0(\sx', i \halfcan')$. Furthermore, we were
able to chose generators such that $x_{7, 1}$ has maximal pole
order at $P_2$ and $x_{7, 2}$ has maximal pole order at $P_3$.

Use the presentation given above with $y _{2, e_2'} :=
x_{7, 1}$ and $y_{3, e_3'} :=  x_{7, 2}$.
We see that $(\sx', \halfcan', \{2, 3\})$ immediately satisfies
\ref{custom:Ad-i} of Definition ~\ref{defn:admissible}. Next we
check \ref{custom:Ad-ii}.  We may also choose
pole orders of the generators such that $y_{i, e_i'}$ is the only
generator lying on the line $-ord_{P_i}(z) = \deg (\frac{3k}{7} z)$
in the $(\deg z, -ord_{P_i}(z))$  lattice and with the other
generators lying below the line as seen in Figure ~\ref{fig:377}
(e.g. the pole orders $(-ord_{P_1}(z)$, $-ord_{P_2}(z)$, $-ord_{P_3}
(z))$ may be chosen to be $(1, 1, 1)$, $(1, 2, 2)$, $(2, 3, 2)$,
and $(2, 2, 3)$ for $z = x_{3, 1}$, $x_{5, 1}$, $x_{7, 1}$, and
$x_{7, 2}$ respectively).

\begin{figure}[H]
\includegraphics{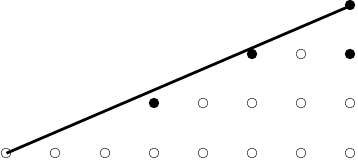} \\
\caption{Generators in the $(\deg z, -ord_{P_2}(z))$ lattice}
\label{fig:377}
\end{figure}

Also note that $e_j' + 2k = 7 + 2k \nmid 2 = e_i - e_j'$ for all
$i, j \in J = \{2, 3\}$ such that $j \neq i$ and for all $k \geq
0$. Thus, $s_{\sigma, J}(i, k) = \emptyset$ for each $i \in J$.
Furthermore, $\deg \lfloor e_i \halfcan \rfloor = \deg \lfloor 9
\halfcan \rfloor = 2 \lfloor \frac{4 \cdot 9}{9} \rfloor + \lfloor
\frac{9}{3} \rfloor - 9 = 2$, so \ref{custom:Ad-iii} is satisfied and
$(\sigma, \sx', \{2, 3\})$ is admissible.
\end{example}

\begin{example}
\label{eg:base-1-33-adm}
Let $(\sx',0 , \halfcan')$ be a log spin curve of genus $1$ with
$\halfcan' = P - Q + \frac{1}{3}P_1 + \frac{1}{3}P_2$, as in Example
~\ref{eg:base-1-33}. Here, we check that $(\sx', \halfcan', \{1,2\})$
is admissible.

Recall that
$$R_{\halfcan'} \cong \Bk[u, x_3, y_3, y_4]/(x_3 y_3- \alpha uy_4, y_4^2 - \beta x_3^2 u - \gamma y_3^2u).$$
We have two generators $x_3$ and $y_3$ in degree 3 with a pole of order $1=\frac{3- 1}{2}$ by construction. Hence, \ref{custom:Ad-i} holds. We next check \ref{custom:Ad-ii} for the point $P_1$, as the case of $P_2$ is symmetric. Here, by construction
\[
\frac{-\ord_{P_1}(z)}{\ord(z)} = \begin{cases}
	0 &\text{ if }z \in \{u, y_3\}\\
	\frac{1}{4} &\text{ if }z = y_4.
\end{cases}
\]

\noindent
Since $0, \frac{1}{4} < \frac{1}{3}$, \ref{custom:Ad-ii} holds.
Finally, to check \ref{custom:Ad-iii}, note that $\max_{k \geq 0}S_{\sigma,J}(i,k) = 0.$ Therefore, $\deg \lfloor 5L \rfloor  = 2 > 1 = (2 g - 1) + 0$.
\end{example}

The following lemma will slightly strengthen condition
\ref{custom:Ad-ii} from Definition ~\ref{defn:admissible}.
This improvement is crucial in the proof of part (c) of Lemma
~\ref{lem:raise-stacky-order}.

\begin{lem}
\label{lem:admissible_inequality}
For any $z$ as in condition \ref{custom:Ad-ii} of definition ~\ref{defn:admissible} the inequality \ref{custom:Ad-ii} implies the
tighter inequality that

\begin{align*}
	-\ord_{Q_i}
^{\halfcan_X'}(z) \leq \deg(z) \subhalf{e_i'} -\frac{1}{e_i'}
\end{align*}
\end{lem}

\begin{proof}
We know by \ref{custom:Ad-ii} that

\begin{align*}
	-\ord_{Q_i}
^{\halfcan_X'}(z) < \deg(z) \subhalf{e_i'}
\end{align*}

\noindent
If we write $\frac{\alpha}{\beta} = \deg(z) \frac{e_i'- 1}{2e_i'}$ 
as a fraction in lowest terms, then we see $\beta \mid e_i'$ since $
e_i'- 1$ is even. Therefore, since $-\ord_{Q_i}
^{\halfcan_X'}(z)$ is an integer, 
we must have

\begin{align*}
	-\ord_{Q_i}
^{\halfcan_X'}(z) \leq \deg(z) \subhalf{e_i'} - \frac{1}{\beta} \leq 
	\deg(z) \subhalf{e_i'} - \frac{1}{e_i'}.
\end{align*}
\end{proof}

\begin{lem}
\label{lem:admissible_subset}
If $(\sx', \Delta, \halfcan')$ is a log spin curve, $(\sx', \halfcan', J)$ is admissible and $W \subseteq J$ is any subset,
then $(\sx', L', W)$ is also admissible.
\end{lem}

\begin{proof}
Each of the conditions \ref{custom:Ad-i}, \ref{custom:Ad-ii}, and \ref{custom:Ad-iii} hold for $W$
if they hold for $J$.
\end{proof}

\begin{rem}
\label{rem:three-cases}
For our inductive arguments in Theorems ~\ref{thm:g-high-main},
~\ref{thm:g-1-main}, and ~\ref{thm:g-0-main}, we will often add in a
single stacky point with stabilizer order $3$. Say $(\sx, \Delta,
\halfcan)$ is a log spin curve with signature $\sigma := (g; e_1,
\ldots, e_r; \delta) = (g; 3, \ldots, 3; \delta)$, and our base cases
include signatures $\sigma$ characterized by one of the following:
cases, which we will soon refer to in Lemma
~\ref{lem:raise-stacky-order}:

\begin{enumerate}
	\item[\customlabel{custom:three-cases-1}{(1)}] $g = 0, e_1 =
		\cdots = e_r = 3, \delta = 0$, and $r \geq 5$
	\item[\customlabel{custom:three-cases-2}{(2)}] $g = 1, e_1 =
		\cdots = e_r = 3, \delta = 0$, and $ r \geq 2$
	\item[\customlabel{custom:three-cases-3}{(3)}] $g \geq 2, e_1 =
		\cdots = e_r = 3, \delta$ arbitrary, and $r \geq 1$.
\end{enumerate}
\end{rem}

\begin{lem}
\label{lem:raise-stacky-order}
Suppose $(\sx', \Delta, \halfcan')$ is a log spin curve with coarse
space $X'$ and signature $\sigma := (g; e_1', \ldots, e_r'; \delta)$.
Define $R':= R_{\halfcan'}$.  Further, assume either 
\begin{enumerate}
	\item $(\sx', \halfcan', J)$ is admissible with generators $x_1,
		\ldots, x_m \in R'$ and $y_{i, e_i'} \in R' $ for all $i \in J$, as 
		in Definition ~\ref{defn:admissible} or
	\item $\sigma$ is one of the signatures described in Cases
		\ref{custom:three-cases-1}, \ref{custom:three-cases-2} and
		\ref{custom:three-cases-3} of Remark ~\ref{rem:three-cases} and
		$y_{1, 3} = y_{2,3} = \ldots= y_{r,3}$ is a rational section of
		$\sco(\halfcan_X)$ with $\ord_{P_i}^{\halfcan_X'}(y_{1,3}) = 1$
		for $1 \leq  i \leq r$.
\end{enumerate}
Let
$(\sx, \Delta, \halfcan)$ be another log spin curve
with coarse space $X$, so that $X \cong X'$ and with signature $(g; e_1, \ldots, e_r;
\delta)$ such that $e_i = e_i' + 2$ for all $i \in J$ and
$e_j = e_j'$ for $j \notin J$. Define $R := R_\halfcan$.  Then the following are true:
\begin{enumerate}
	\item[(a)] For all $i \in J$, there exists $y_{i, e_i} \in
		H^0(\sx, e_i(K_\sx))$ so that
		\begin{align*}
			-\ord_{Q_i}
^{\halfcan_X'}(y_{i, e_i}) = \frac{e_i - 1}{2}
		\end{align*}
		and
		\begin{align*}
			\frac{-\ord_{Q_j}
^{\halfcan_X'}(y_{i, e_i})}{\deg (y_{i, e_i})} \leq 
\subhalf{
			e_j'} - \frac{1}{\deg(y_{i, e_i})e_j'}
		\end{align*}
		for all $j \in J$ with $j \neq i$.
	\item[(b)] A choice of elements $y_{1, e_1}, \ldots, y_{r, e_r}$ as in part (a) minimally
		generate $R$ over $R'$.
	\item[(c)] Endow $\Bk[y_{i, e_i}]_{i\in J}$ and $ \Bk[x_1, \ldots, x_m, y_{i, e_i'}]_{i\in J}$ with graded monomial orders and give 
$\Bk[y_{i, e_i}]_{i\in J} \otimes \Bk[x_1, \ldots, x_m, y_{i, e_i'}]_{i\in J}$
 block order.  Let $I$ be the kernel of $\Bk[x_1, \ldots, x_m, y_{i, e_i'}, y_{i, e_i}]_{i\in J} \to R$.
  Then,
		\begin{align*}
			\initial(I) \;	&= \; \initial(I')\Bk[x, y] \\
					&+ \; \langle y_{j, e_j}x_i  : 1\le i \le m, j\in J\rangle \\
					&+ \; \langle y_{j, e_j}y_{i, e_i'}: i,j\in J, i\ne j\rangle \\
					&+ \; \langle y_{j,e_j}y_{i,e_i}: i,j\in J, i\ne j\rangle.
		\end{align*}
	\item[(d)] The triple $(\sx, \halfcan', J)$ is admissible.
\end{enumerate}
\end{lem}

\sssec*{Idea of Proof:}
The construction of the $y_{i, e_i}$ in Part (a) uses Riemann--Roch
and condition ~\ref{custom:Ad-iii}. The inequality in Part (a)
follows from the fact that $\halfcan$ is similar to $\halfcan'$,
but with some of the $e_i$ incremented by $2$. Part (b) follows
because the sub-lattice spanned by $y_{i, e_i}$ and $y_{i, e_i'}$
has determinant $1$ for $i \in J$, so the generators in Part (a)
generate all of $R_\halfcan$ over $R_{\halfcan'}$. To check part (c),
we construct relations between the generators of $R_{\halfcan}$
over $R_{\halfcan'}$: we note that an element lies in $R_{\halfcan'}$
if and only if its pole order at $P_i$ for $i \in J$ is not too
large, and use Lemma ~\ref{lem:admissible_inequality} to bound pole
orders. Part (d) follows fairly easily from the definition of
admissibility.

\begin{proof}
We prove this in case (1). Case (2) follows a similar procedure.

{\bf Part (a):} 
By Definition ~\ref{defn:admissible}, for all $i\in J$

\begin{align*}
	S(i,0) = \{j \in J : j \neq i \text{ and }e_j' \mid e_i-e_j'\} = \{j \in J : j \neq i \text{ and }e_j' \mid e_i\}.
\end{align*}

\noindent
Define
\begin{align*}
	E_i = \sum_{j \in S(i,0)}^{}Q_j.
\end{align*}

\noindent
The assumption \ref{custom:Ad-iii} implies
\begin{align*}
	\deg \left( e_i L' - E_i \right) \geq \max(2g - 1,0),
\end{align*}

\noindent
and so $H^0(\sx', e_iL'-E_i + Q_i)$ is base point free by 
Riemann--Roch.
Hence, a general element
\begin{align*}
	y_{i, e_i} \in H^0(\sx', e_iL'-E_i + Q_i)
\end{align*}

\noindent
satisfies
\begin{align*}
	-\ord_{Q_i}
^{\halfcan_X'}(y_{i, e_i}) = \left\lfloor e_i \subhalf {e_i'} \right\rfloor + 1 =
	\frac{e_i - 1}{2}.
\end{align*}

Noting that
\begin{align*}
	\lfloor e_i L' \rfloor + Q_i \leq \lfloor e_i L \rfloor,
\end{align*}

\noindent
we obtain an inclusion
\begin{align*}
	H^0(\sx', e_iL'-E_i + Q_i) \rightarrow H^0(\sx, e_iL - E_i) \subseteq 
H^0(\sx, e_iL)
\end{align*}

\noindent
meaning that $y_{i,e_i}\in H^0(\sx,e_iL)$ satisfies the first part of claim (a).

We next show $y_{i, e_i}$ also satisfies the second part of the
claim of $(a),$ by considering separately the cases in which $j
\in S(i, 0),$ and $j \notin S(i,0)$.

If $j \in S(i,0)$, then $E_i \geq Q_j$ gives $y_{i, e_i} \in H^0
(\sx', e_iL'-E_i + Q_i)$ an extra vanishing condition at $Q_j$, so
\begin{align*}
	-\ord_{Q_j}
^{\halfcan_X'}(y_{i, e_i}) \leq e_i\subhalf {e_j'} - 1 \leq e_i 
	\subhalf{e_j'} - \frac{1}{e_j'}.
\end{align*}

If instead $j\ne i$ and $j \notin S(i,0)$, then since $e_j' \nmid e_i$, we know
$e_i\subhalf{e_j'} \notin \BZ$, so
\begin{align*}
	-\ord_{Q_j}
^{\halfcan_X'}(y_{i, e_i}) \leq \left\lfloor  e_i\subhalf{e_j'} \right\rfloor 
	\leq e_i\subhalf{e_j'} - \frac{1}{e_j'},
\end{align*}

\noindent
completing the proof of (a).

{\bf Part (b):}
Define $R_0 = R'$ and for $i \in \{1, \ldots, r\},$ inductively define 
$$R_i = \begin{cases}
	R_{i - 1} &\text{ if }i \notin J\\
	R_{i - 1}[y_{i, e_i}] &\text{ if }i \in J. 
\end{cases}$$

\noindent
To prove (b), it suffices to show that elements of the form $y_{i, e_
i'}^ay_{i, e_i}^b$ with $a \geq 0, b > 0$ form a $\Bk$-basis for $R_{i}$ over $R_{i-1}$. These elements do not lie in $R_{i - 1}$ because the pole order of $y_{i, e_i'}^ay_{i,e_i}^b$ at $Q_i$ is larger than that of any element in the $k^{th}$ component of $R_{i - 1}$. 
Additionally, these elements are linearly independent amongst themselves because of
injectivity of the linear map

\begin{align*}
	(a,b) \mapsto \left( \deg\left(y_{i, e_i'}^ay_{i, e_i}^b\right),-
	\ord_{Q_i}
^{\halfcan_X'}\left( y_{i, e_i'}^ay_{i, e_i}^b \right)  \right) = (a,b) 
	\begin{pmatrix}
		e_i -2 & \frac{e_i -3}{2} \\
		e_i	 & \frac{e_i - 1}{2}
	\end{pmatrix}.
\end{align*}

\noindent
Furthermore, $\{y_{i, e_i'}^a y_{i, e_i}^b:a \geq 0,
b > 0\}$ span $R_i$ over $R_{i - 1}$, because the set of
integer lattice points in the cone generated by the vectors $\left(e
_i -2, \frac{e_i -3}{2} \right)$ and $\left(e_i, \frac{e_i - 1}{2}
\right)$ is saturated, because the corresponding determinant is
\begin{align*}
	(e_i -2) \frac{e_i - 1}{2} - e_i \frac{e_i -3}{2} = 1.
\end{align*}
This completes part (b).

{\bf Part (c):}
To show (c), we wish to show that $y_{i, e_i}z \in R'$ for all generators $z$ of $R'$
with $z \neq y_{i, e_i} \text{ and } z \neq y_{i, e_i'}$. By definition of $H^0(\sx,\halfcan')$, note that $f \in R$ further satisfies $f \in R'$ if and
only if for all $j \in J$ we have
\begin{align}
\label{eqn:order-degree}
	-\ord_{Q_j}
^{\halfcan_X'}(f) \leq \deg (f) \left( \subhalf {e_j'} \right).
\end{align}

\noindent
Now fix $i \in J$. We check that Inequality ~\eqref{eqn:order-degree} holds with $f = y_{i, e_i}z$, implying $y_{i,e_i}z \in R'$ in the three following cases:

\vskip.1in
\noindent
{\sf Case 1: $j \notin \{i \} \cup S(i,0).$}

Here, $L|_{Q_j} = L'|_{Q_j},$ so
\begin{align*}
	-\ord_{Q_j}
^{\halfcan_X'}(y_{i, e_i})-\ord_{Q_j}
^{\halfcan_X'}(z) \leq e_i \subhalf {e_j'} +
	\deg (z) \subhalf {e_j'} = \deg (y_{i, e_i}z) \subhalf{e_j'}. 
\end{align*}

\vskip.1in
\noindent
{\sf Case 2: $j =i.$}

By part (a), condition \ref{custom:Ad-ii}, and Lemma
~\ref{lem:admissible_inequality} we have
\begin{align*}
	-\ord_{Q_j}
^{\halfcan_X'}(y_{i, e_i})-\ord_{Q_j}
^{\halfcan_X'}(z)
	&\leq \frac{e_j - 1}{2} + \deg (z) \left( \subhalf{e_j'} \right) -\frac{1}{e_j'} \\
	& = \frac{e_j - 1}{2} -\frac{1}{e_j'} +\deg (z) \left( \subhalf{e_j'} \right) \\
	&= \frac{e_j(e_j -3)}{2(e_j -2)} +\deg (z) \left( \subhalf{e_j'} \right) \\
	&= \deg (y_{i, e_i}z) \subhalf{e_j'}.
\end{align*}

\vskip.1in
\noindent
{\sf Case 3: $j \in S(i,0).$}

In this case, we may first assume $z \neq y_{j, e_j}$, 
as this is covered by case 2, with $i$ and $j$ reversed. 
Hence, 
\begin{align*}
	-\ord_{Q_j}^{\halfcan_X'}(z) \leq \deg z \subhalf{e_j'},
\end{align*}
implying
\begin{align*}
	-\ord_{Q_j}^{\halfcan_X'}(y_{i, e_i})-\ord_{Q_j}^{\halfcan_X'}(z) 
	&\leq  e_i \subhalf{e_j'} + \deg z \subhalf{e_j'} 
	= \deg (y_{i, e_i}z) \subhalf{e_j'},
\end{align*}
completing part (c).

{\bf Part (d):}
To check (d), we show \ref{custom:Ad-i}, \ref{custom:Ad-ii}, and \ref{custom:Ad-iii} are satisfied. We know \ref{custom:Ad-i} holds by part (b), taking the $y_{
i, e_i}$ as the generators in degree $e_i$. Next, \ref{custom:Ad-ii} is
strictly monotonic in the $e_i$ and hence also holds for $(\sx, J)$.
Finally, if \ref{custom:Ad-iii} holds for $e$ then it holds for $e + 2$ by
definition. This is where we use that \ref{custom:Ad-iii} holds for $k > 0$
and not just for $k = 0$.
\end{proof}

\begin{cor}
\label{cor:raise-stacky-order}
Suppose $(\sx', \halfcan',J')$ is admissible with signature $\sigma' = (e_1
', \ldots,e_r')$ or $\sigma$ satisfies one of the conditions of
Remark ~\ref{rem:three-cases}. Let $J' = \{t,t+1, \ldots,r\}$ and
$e_1' \leq e_2' \leq \cdots \leq e_t' = e_{t+1}' =\cdots = e_r'$,
so that $(\sx', \Delta, L)$ satisfies the conditions of Lemma
~\ref{lem:raise-stacky-order} and Theorem ~\ref{thm:main}. Then,
for any spin curve $(\sx, \Delta, L)$ so that $\sx$ and $\sx'$ have
the same coarse space $X = X'$ with the same set of stacky points,
and $  \sx$ has signature $(g; e_1, \ldots, e_r; \delta)$ with $e_1
\leq e_2 \leq \cdots e_r$ so that
\begin{align*}
	& e_i	= e_i' &\text{ if }i \notin J \\
	& e_i \ge e_i' &\text{ if } i \in J
\end{align*}

\noindent
then Theorem ~\ref{thm:main} holds for $(\sx, \Delta, \halfcan)$.
\end{cor}

\begin{proof}
For $t \leq i \leq r$, let $(\sx_i, \Delta, \halfcan_i)$ be the log spin curve with coarse space $X_i$ so that $X_i = X',$ with the same stacky points as $(\sx', \Delta, \halfcan),$ and having signature $(g; e_1, \ldots. e_{i - 1}, e_i, e_i, \ldots, e_i; \delta).$ Let 
$J_i = \{i, \ldots, r\}.$ Note that $(\sx_0, \halfcan_0,J_0) = (\sx', \halfcan',J')$ and $(\sx_r , \halfcan_r ,J_r) = (\sx, L, \{r\})$.

Let $(*_i)$ denote the condition that $(\sx_i, \Delta, L_i)$
satisfies the conditions of Lemma ~\ref{lem:raise-stacky-order} and
Theorem ~\ref{thm:main}, and $(\sx_i, \halfcan_i, J_i)$ is
admissible. Since $(*_0)$ holds by assumption, it suffices to show
that if $(*_i)$ holds then so does $(*_{i + 1})$. Indeed, $(\sx_i,
\halfcan_i, J_{i + 1})$ is admissible by an application of Lemma
~\ref{lem:admissible_subset} and the fact that $J_{i + 1} \subseteq J
_i$. Then, applying Lemma ~\ref{lem:raise-stacky-order} with the
fixed set $J_{i + 1}$ repeatedly ($\frac{e_{i + 1} - e_i}{2}$ many
times) yields $(*_{i + 1})$.
\end{proof}

%%%%%%%%%%%%%%%%%%%%%%%%%%%% Genus $\geq 2$ %%%%%%%%%%%%%%%%%%%%%%%%%%%%%%%

\section{Genus At Least Two}
\label{sec:g-high}
We now consider the case when the genus is at least 2. In this case, we are able to bound the degrees of generators of $R_\halfcan$ and its ideal of relations. In this section, we do not obtain explicit presentations of $R_\halfcan$. This contrasts with Sections ~\ref{sec:g-1} and ~\ref{sec:g-0} where we not only obtain bounds, but also obtain inductive presentations. The tradeoff is that in the genus zero and genus one cases, we have to deal with explicit base cases. In this section we apply general results.

\ssec{Bounds on Generators and Relations in Genus At Least Two}
\label{ssec:bounds-high-genus}

The main result of this subsection is that for a log spin curve with no stacky points $(X, \Delta, L)$, the spin canonical ring $R_\halfcan$ is generated in degree at most $5$, with relations in degree at most $10$. The case that $\Delta = 0$ was completed by Reid \cite[Theorem 3.4]{reid:infinitesimal}. For $\Delta > 0$, the generation bound is shown in Lemma ~\ref{lem:generation-5} and the relations bound is shown in Lemma ~\ref{prop:relation-10}. Throughout this subsection, we will implicitly use Remark ~\ref{rem:delta-not-1}, which implies $\deg 2L \geq 2g$ so $H^0(X, 2L)$ is basepoint-free by Riemann--Roch. We summarize the results of this subsection in Corollary ~\ref{cor:g-2-presentation-bound}.

The proofs of this subsection results are similar to those in Neves ~\cite[Proposition III.4 and Proposition III.12]{neves:halfcan}. However, the statements differ, as we assume $\Delta > 0$ instead of $\Delta = 0$ and do not assume there is a basepoint-free pencil contained in $H^0(X, L).$ 

\begin{lem}
\label{lem:generation-5}
Let $(X, \Delta, L)$ be a log spin curve of signature $(g;-;\delta),$ {\rm(}where $-$ means $X$ is a bona fide scheme and has no stacky points,{\rm)}
with $g \geq 2$ and $\Delta > 0.$ Let $s_1$, $s_2 \in H^0(X, 2L)$
be two independent sections such that the vector subspace $V =
\newspan(s_1, s_2) \subseteq H^0(K)$ is basepoint-free. Then, the map
\begin{align*}
	V \otimes H^0(nL) \rightarrow H^0((n+2)L)
\end{align*}
is surjective if $n \geq 4$. In particular, $R_\halfcan$ is generated in degree at most $5.$
\end{lem}

\begin{proof}
To show there are no new generators in degree at least $6$, it
suffices to show that if $n \geq 4$, the map
\begin{align*}
	H^0(nL) \otimes H^0(2L) \rightarrow H^0((n+2)L)
\end{align*}
is surjective. Indeed, since $V = \newspan(s_1, s_2) \subseteq
H^0(K)$ is basepoint-free, by the basepoint-free pencil trick (see
\cite[Lemma 2.6]{saint-donat:proj} for a proof), we obtain an exact
sequence
\begin{align*}
0 \longrightarrow H^0((n - 2)\halfcan) \longrightarrow V \otimes
H^0(n \halfcan) \overset{f}{\longrightarrow} H^0((n + 2)\halfcan)
\end{align*}

We wish to show $f$ is surjective.
Note that $\dim_\Bk \ker f = \dim_\Bk H^0((n - 2)L) = (n - 3)(g - 1+\frac{\delta}{2})$ using Riemann--Roch and the assumption $n \geq 4.$
Additionally, $\dim_\Bk V \otimes H^0(nL) = 2 \cdot (n - 1)(g - 1+\frac{\delta}{2})$, again using Riemann--Roch.
Therefore,
\begin{align*}
	\dim_\Bk \im f 	& = 2 \cdot (n - 1)(g - 1+ \frac{\delta}{2}) - (n - 3)(g - 1 + \frac{\delta}{2})\\
				& = (n + 1)(g - 1 + \frac{\delta}{2}) = \dim_\Bk H^0((n + 2)L).
\end{align*}
Ergo, $f$ is surjective.
\end{proof}

The next step is to bound the degrees of the relations of $R_\halfcan$ when $\Delta > 0$. This is done in Proposition ~\ref{prop:relation-10}
by using the basepoint-free pencil trick to show that if a relation lies in a sufficiently high degree, it lies in the ideal generated by the relations in lower degrees.  In Definition ~\ref{defn:lower-ideal}, we fix notation for the ideal generated by lower degrees relations:

\begin{defn}
\label{defn:lower-ideal}
Let $(X, \Delta, L)$ be a log spin curve of signature $(g; -; \delta
)$ with $g \geq 2$ and $\Delta > 0.$ Choose generators $x_1, \ldots
, x_n$ of $R_\halfcan$ so that we obtain a surjection $\phi: \Bk[x_1, \ldots
, x_n] \twoheadrightarrow R_\halfcan$ with kernel $I_\halfcan$. Let $I_{\halfcan, k}$ be
the $k^\text{th}$ graded piece of $I_\halfcan$ and define
\[
	J_{\halfcan,k} = \sum_{j = 1}^{k - 1}\Bk[x_1, \ldots, x_n]_j \cdot I_{\halfcan,k-j}.
\]
\end{defn}

\begin{lem}
\label{lem:reducing-degree}
Let $(X, \Delta, \halfcan)$ be a log spin curve of genus $g \geq 2,$ so
that $\Delta > 0.$ Choose generators $x_1, \ldots, x_n$ of $R_\halfcan$ so
that we obtain a surjection $\phi: \Bk[x_1, \ldots, x_n]
\twoheadrightarrow R_\halfcan$ with kernel $I_L$. Let $s_1, s_2 \in \Bk[x_1,
\ldots, x_n]_2$ be two elements so that $\: \newspan(\phi(s_1), \phi(s_2))$
$= V \subseteq H^0(X, 2\halfcan)$ is basepoint-free. For any $f \in
\Bk[x_1, \ldots, x_n]$ such that $\deg f \geq 11,$ there exist $g, h \in
\Bk[x_1, \ldots, x_n]_{k - 2}$ so that $s_1 g + s_2 h \equiv f \bmod J_{\halfcan, k}.$
\end{lem}
\begin{proof}
By Lemma ~\ref{lem:generation-5}, $\deg x_i \leq 5$ for $1 \leq i \leq n$. Therefore, we may write $f = \sum_{i = 1}^{n}a_i x_i$ with $a_i \in \Bk[x_1, \ldots, x_n]_{k-\deg x_i}$.
We next show that for all $1 \leq i \leq n$ there exist $g_i, h_i \in \Bk[x_1, \ldots, x_n]_{k - \deg x_i - 2}$ so that $a_i = s_1g_i + s_2h_i \bmod I_{\halfcan, \deg a_i}.$ 

By Lemma ~\ref{lem:generation-5},
\begin{align*}
	V \otimes H^0((\deg f-\deg x_i -2)\halfcan) \rightarrow H^0((\deg f-\deg x_i)\halfcan)
\end{align*}
is surjective because $\deg f \geq 11$ implies that
$$\deg a_i =\deg f - \deg x_i -2 \geq 4.$$
In particular, there exist $\alpha_i, \beta_i \in R_\halfcan$ so that
$$\phi(a_i) = \phi(s_1) \cdot \alpha_i + \phi(s_2) \cdot \beta_i.$$ 
Choosing $g_i,h_i\in \Bk[x_1, \ldots, x_n]_{\deg(a_i)-2}$ for $1 \leq i \leq n$ so that $\phi(g_i) = \alpha_i, \phi(h_i) = \beta_i$, we have 
$$a_i \equiv s_1 g_i + s_2 h_i \bmod I_{\halfcan, \deg a_i},$$ as claimed.

Finally, we may then take $g = \sum_{i}^{}g_i x_i, \: h = \sum_{i}^{}h_i x_i,$ so that 
\begin{align*}
	f &\equiv \sum_{i}^{}a_i x_i \equiv \sum_{i}^{}(s_1g_i + s_2h_i)x_i \equiv s_1 \left( \sum_{i}^{}g_i x_i \right) + s_2 \left( \sum_{i}^{}h_i x_i \right) \\
	&\equiv s_1 g + s_2 h \bmod J_{L,k}.
\end{align*}
\end{proof}

\begin{prop}
\label{prop:relation-10}
Let $(X, \Delta, \halfcan)$ be a log spin curve of signature $(g; -; \delta)$
 with $g \geq 2$ and $\Delta > 0.$ Then $I_\halfcan$ is generated in
degree at most $10$.
\end{prop}
\begin{proof}
Suppose $f \in I_\halfcan$ with $\deg f \geq 11$. To complete the proof,
it suffices to show $f \in J_{\halfcan,k}$. By Lemma
~\ref{lem:reducing-degree}, this is the same as checking $s_1 g + s_2 h \in
J_{\halfcan, k}$ where $\phi(s_ 1), \phi(s_2) \in H^0(X, 2L)$ are two sections
so that $\newspan(\phi(s_ 1), \phi(s_2)) = V \subseteq H^0(K)$ is
basepoint-free. Consider the map

$$\begin{tikzcd}
V \otimes H^0((\deg f - 2)L) \ar {r}{f} & H^0((\deg f)L),
\end{tikzcd}$$

\noindent
we know that $\phi(s_1)\phi(g) + \phi(s_2) \phi(h) \mapsto 0$.
So by the explicit isomorphism given in the proof of the
basepoint-free pencil trick, as shown in the proof of \cite[Lemma 2.6]
{saint-donat:proj}, there exists some $\rho \in \Bk[x_1, \ldots, x_n]$
so that $\phi(\rho) \in H^0((\deg f - 4)L)$ satisfies $\phi(g) =
\phi(s_2) \phi(\rho)$ and $\phi(h) = -\phi(s_1)\phi(\rho)$. 
Therefore,
$g \equiv s_2 \rho \bmod I_{\halfcan, k-2}$ and $h \equiv - s_1 \rho \bmod
I_{\halfcan, k - 2}$. Hence,
\begin{align*}
	s_1g + s_2h \equiv s_1(s_2\rho) + s_2(-s_1 \rho) \equiv 0 \bmod J_
{L,k}.
\end{align*}
\end{proof}

We now summarize what we have shown.

\begin{cor}
\label{cor:g-2-presentation-bound}
Let $(X, \Delta, L)$ be a log spin curve of signature $(g; -;\delta)$
with $g \geq 2$. Then $X$ has minimal generators in degree at most
$5$ and minimal relations in degree at most $10$.
\end{cor}

\begin{proof}
If $\delta = 0$, the result is immediate from Reid \cite[Theorem 3.4]
{reid:infinitesimal}. Otherwise, if $\delta > 0$, the bound on the
degrees of minimal generators follows from Lemma
~\ref{lem:generation-5}, while the bound on the degrees of minimal
relations follows from Proposition ~\ref{prop:relation-10}.
\end{proof}

\ssec{Main Theorem for Genus At Least Two}
\label{ssec:g-high-main}

We are ready to prove our main theorem, Theorem ~\ref{thm:main} in the case
$g \geq 2$. The idea of the proof is to use Corollary \ref{cor:g-2-presentation-bound} to 
complete the base case when $L = \halfcan_X$ and then apply Lemma
~\ref{lem:sat-1}, Lemma ~\ref{lem:sat-2}, and Lemma
~\ref{lem:raise-stacky-order} to complete the induction step.

\begin{thm}
\label{thm:g-high-main}
Let $g \geq 2$ and let $(\sx, \Delta, \halfcan)$ be a log spin curve
with signature $(g; e_1, \ldots, e_r; \delta)$. Then the
log spin canonical ring $R(\sx, \Delta, \halfcan)$
is generated as a $\Bk$-algebra by elements in degree at most $e =
\max(5, e_1, \ldots, e_r)$ with minimal relations in degree at most $2e$.
\end{thm}

\begin{proof}
As the base case, let $\halfcan = \halfcan_X \in \di X$ satisfy $2 \halfcan \sim 2 K_X + \Delta$.
By Corollary ~\ref{cor:g-2-presentation-bound}, the theorem holds for $(\sx, \Delta,
\halfcan_X)$. 

Next, suppose the theorem holds for $\halfcan' = \halfcan_X + \sum_{i=1}^{r-1} \frac{1}{3}P_i$. Let $\halfcan$ be a log spin canonical divisor of the form $\halfcan = \halfcan_X + \sum_{i = 1}^r \frac{1}{3} P_i$, which means that $\halfcan = \halfcan' + \frac{1}{3} P_r$.  If $P_r$ is a basepoint of $\halfcan'$, then the theorem holds for $\halfcan$ by Lemma ~\ref{lem:sat-2}.
Otherwise, $P_r$ is not a basepoint of $\halfcan'$, meaning that in particular $R_{\halfcan'}$ is saturated in 1.
Therefore, since $P_r$ is a not basepoint of $\halfcan'$, Equation
~\ref{eqn:deg1-sat-ind} holds by Riemann--Roch.
In this case, the theorem holds for $\halfcan$ by Lemma ~\ref{lem:sat-1}.

We have thus shown the theorem for all $(\sx, \Delta, \halfcan)$ with $g \ge 2$ and signature $(g; 3, \ldots, 3; \delta)$. Therefore, by Corollary 
~\ref{cor:raise-stacky-order}, this theorem 
holds for all log spin curves $(\sx, \Delta, \halfcan)$. 
\end{proof}

\begin{rem}
\label{rem:genus-2-explicit-bound}
For this remark, we retain the terminology from the proof of Theorem ~\ref{thm:g-high-main}.
Suppose $e :=  \max(e_1,\ldots, e_r).$ Then, $R_\halfcan$ has a generator in degree $e$ when $e \geq 5$ and a relation in degree at least $2e-4$ when $e \geq 7$. Since the proof of Theorem ~\ref{thm:g-1-main} is given by inductively 
applying Lemmas ~\ref{lem:sat-1} and ~\ref{lem:sat-2}, we obtain that $R_\halfcan$ is minimally generated over $R_{L'}$ by an element in degree $e,$ assuming $e \geq 5$. Furthermore, if $e \geq 7$ and $e_i = e,$ 
then there must be a relation with leading term $y_{i,e_i} \cdot y_{i,e_i-4}.$ Hence, there is a relation in degree at least $2e - 4$. 
Further, by examining the statements of Lemma ~\ref{lem:sat-1} and \ref{lem:sat-2} in the case that there are $1 \leq i < j \leq r$ with $e_i = e_j = e$, then, there is necessarily a relation with leading term $y_{i, e_i} \cdot y_{j, e_j}$ in degree $2e$. This analysis also applies to the cases that $g = 0$ and $g = 1$.
\end{rem}

%%%%%%%%%%%%%%%%%%%%%%%%%%%% Genus one %%%%%%%%%%%%%%%%%%%%%%%%%%%%%%%

\section{Genus One}
\label{sec:g-1}
In this section, we prove Theorem ~\ref{thm:main} in the case that $g = 1$.  We follow a similar inductive strategy as in the genus $g\geq 2$ case, except unlike in the $g \geq 2$ case we obtain explicit generators and relations here.

In the case of a genus 1 curve, $X$ with no stacky points, we know $K_X \sim 0$, and therefore the only possibilities for log spin canonical divisors are $\halfcan' \sim 0$ or $\halfcan' \sim P - Q$ where $P, Q$ are distinct points, fixed  under the hyperelliptic involution. We inductively construct presentations by adding points through Lemmas ~\ref{lem:sat-1} and ~\ref{lem:sat-2} and incrementing the values of the $e_i$'s using Lemma ~\ref{lem:raise-stacky-order}.

\ssec{Genus One Base Cases}
\label{ssec:g-1_base}
In this subsection we set up the base cases needed for our inductive approach, of proving Theorem ~\ref{thm:main} in the case $g = 1$.

\begin{table}	
\begin{tabular}
{| c || c | c | c |}	
	\hline
	$\halfcan'$ & Generator Degrees & Degrees of Minimal Relations & $e$ \\
	\hline
	\hline
	$0$ & $\{1\}$ & $\emptyset$ & $1$ \\	\hline

	$\frac{3}{7} P_1$ & $\{1,5,7\}$ & $\{15\}$ & $7$ \\	\hline
	
	$\frac{1}{3} P_1 + \frac{1}{3} P_2$ & $\{1, 3, 3\}$ & $\{6\}$ & $5$ \\	\hline
	
	$P - Q + \frac{2}{5} P_1$ & $\{2, 3, 5\}$ & $\{12\}$ & $5$ \\	\hline
	
	$P - Q + \frac{1}{3} P_1 + \frac{1}{3}P_2$ & $\{2, 3, 3, 4\}$ & $\{6,8\}$ & $5$ \\	\hline
\end{tabular}	

\caption{Genus 1 Base Cases}
\label{table:g-1-base}
\end{table}

Generators and relations for $R_{L'}$ with $L' = 0, \: L' = P-Q +\frac{1}{3}P_1 + \frac{1}{3}P_2,$ and $L' = P-Q + \frac{2}{5}P_1$ were checked in Examples ~\ref{eg:base-1-0}, \ref{eg:base-1-33}, and \ref{eg:exception-1-5} respectively. Note that admissibility for $(\sx',0,P-Q +\frac{1}{3}P_1 + \frac{1}{3}P_2)$ is verified in Example ~\ref{eg:base-1-33-adm}. The verification of admissibility for the other cases is similar.

The base cases of $L' = \frac{3}{7}P_1$ and $L' = \frac{1}{3}P_1 +\frac{1}{3}P_2$ can be similarly computed. Although $L' = P - Q + \frac{2}{5}P_1,$ and $L' = \frac{3}{7}P_1$ are used as base cases for the induction, they are also exceptional cases; see Table ~\ref{table:g-1-exceptional}.

\ssec{Genus One Exceptional Cases}
\label{ssec:g-1-exceptional}
Let $\sx$ be a stacky curve, with $P, Q$ distinct hyperelliptic
fixed points on $X$. The following table provides a list of all
cases which are not generated in degrees $e := \max(5, e_1, \ldots,
e_r)$ with relations in degrees $2e$, as described in Theorem
~\ref{thm:g-1-main}.

\begin{table}	
\begin{tabular}
{| c || c | c | c |}
	\hline
	$\halfcan'$ & Generator Degrees & Degrees of Minimal Relations & $e$ \\
	\hline
	\hline
	$P - Q$ & $\{2\}$ & $\emptyset$ & $1$\\	\hline

	$P - Q + \frac{1}{3} P_1$ & $\{2, 3, 7\}$ & $\{14\}$ & $5$ \\	\hline

	$P - Q + \frac{2}{5} P_1$ & $\{2, 3, 5\}$ & $\{12\}$ & $5$\\	\hline
	
	$\frac{1}{3} P_1$ & $\{1, 6, 9\}$ & $\{18\}$ & $5$ \\	\hline

	$\frac{2}{5} P_1$ & $\{1, 5, 8\}$ & $\{16\}$ & $5$ \\	\hline
	
	$\frac{3}{7} P_1$ & $\{1, 5, 7\}$ & $\{15\}$ & $7$ \\	\hline
\end{tabular}

\caption{Genus 1 Exceptional Cases}
\label{table:g-1-exceptional}
\end{table}

We have already checked the case of 
$L' = P - Q + \frac{1}{5}P_1$ above in Example ~\ref{eg:exception-1-5}. The other cases are similar.

\ssec{Main Theorem for Genus One}
\label{ssec:main_g_1}

We now have all the tools necessary to prove our main theorem, ~\ref{thm:main} in the case $g = 1$. 

\begin{thm}
\label{thm:g-1-main}
Let $(\sx, \Delta, \halfcan)$ log spin curve with
signature $\sigma := (1; e_1, \ldots, e_r; \delta)$.
If $g = 1$, then the log spin canonical ring $R(\sx, \Delta, \halfcan)$
is generated as a $\Bk$-algebra by elements of degree at most $\max(5, e_1, \ldots, e_r)$ and has relations in degree at most $2e,$ so long as $\sigma$ does not lie in a
finite list of exceptional cases, as listed in Table
~\ref{table:g-1-exceptional}.
\end{thm}

\sssec*{Idea of Proof:}
We check the theorem in two cases, depending on if $\delta > 0$.
If $\delta > 0$, we first check that the theorem holds for
$\halfcan_X$ by inductively adding in log points to the base case of $
\halfcan_X' = 0$. Then we check that the theorem holds for $\halfcan$
by adding stacky points and then inductively raising the 
stabilizer orders of stacky points in the following sequence of 
steps. When raising the stacky orders, it is important that we 
increment the stabilizer orders of as many stacky points as 
possible to maintain admissibility for the maximal possible sets of 
stacky points. Then, we may also use the fact that raising the 
stabilizer orders of any subset of these maximal sets of stacky 
points will still preserve admissibility.

The check for $\delta = 0$ is similar, although in this case we do
not need to add in log points, only stacky points, and we will need
to utilize the base cases from Subsection ~\ref{ssec:g-1_base}.

\begin{proof}

\vskip.1in
\noindent
{\sf Case 1: $\delta > 0$}

If $\delta > 0$, we must have $\delta \geq 2$, by Remark ~\ref{rem:delta-not-1}. In this case, let $\halfcan_X \in \di X$ satisfy $2 \halfcan_X \sim K_X + \Delta$. We have $\deg \halfcan_X \geq 1$, so, by Riemann--Roch, $h^0(\sx, \halfcan) \geq 1.$ Therefore, $\halfcan$ is linearly equivalent to an effective divisor. Thus, without loss of generality, we may assume $\halfcan$ is an effective divisor.

We first now show the theorem holds for $\halfcan_X$ by induction. Since $\halfcan_X$ is effective, we may induct on the degree of the log spin canonical divisor. The base case is easy: the theorem holds for $\halfcan_X' = 0$ by Example ~\ref{eg:base-1-0}. Assume it holds for $\halfcan_X' \in \di X$, with $\halfcan_X'$ effective. We will show it holds for $\halfcan_X' + P$, verifying the inductive step. There are two cases, depending on whether $P$ is a basepoint of $\halfcan_X'$.

First, if $P$ is not a basepoint of $\halfcan',$ then the hypotheses of Lemma ~\ref{lem:sat-1} are satisfied.
Therefore, by Lemma ~\ref{lem:sat-1}, the theorem holds for $\halfcan = \halfcan' + P$.

Otherwise, $P$ is a basepoint of $\halfcan_X'$, so the hypotheses of Lemma ~\ref{lem:sat-2} are satisfied since $\deg 3(\halfcan_X') \geq 2$ as $\deg \halfcan_X' \geq 1.$ Therefore, by Lemma ~\ref{lem:sat-2}, the theorem holds for $\halfcan_X = \halfcan_X'+P$. By induction, the theorem holds for $\halfcan_X$.

To complete the case that $\delta > 0$, we now need show the theorem holds for a \emph{stacky} log spin canonical divisor $\halfcan$.  It suffices to show that if the theorem holds for a log spin canonical divisor $\halfcan'$ with $\deg \lfloor L' \rfloor > 0$, then it holds for $L' + \subhalf {e_i}P_i$ with $e_i$ odd.  As above, if $P$ is not a basepoint of $\halfcan'$ then the theorem holds for $L' + \subhalf {e_i}P_i$ by Lemma ~\ref{lem:sat-1}. On the other hand, if $P$ is a basepoint of $\halfcan'$ then the theorem holds for $L' + \subhalf {e_i}P_i$ by Lemma ~\ref{lem:sat-2}.

\vskip.1in
\noindent
{\sf Case 2: $\delta = 0$}

Since $\delta = 0$, we may write $\halfcan = \halfcan_X + \sum_{i = 1}^{r}\subhalf{e_i}P_i.$ There are now two further subcases, depending on whether $\halfcan_X = 0$ or $\halfcan_X = P - Q$ for $P$ and $Q$ two distinct hyperelliptic fixed  points.

\vskip.1in
\noindent
{\sf Case 2a: $\halfcan_X = P - Q, P \neq Q$}

Note that we are assuming $\halfcan$ is not one of the exceptional cases listed in Table ~\ref{table:g-1-exceptional}, so we may either assume $\sx$ has 1 stacky point with $e_1 > 5$ or at least 2 stacky points.

First, we deal with the case $\sx$ has at least 1 stacky point. By Example ~\ref{eg:exception-1-5}, if $L' = P - Q + \frac{2}{5}P_1,$ then $R_{\halfcan'}$ is generated in degrees $2,3,$ and $5$ with a single relation in degree 12. Furthermore, $(\sx, \halfcan', \{1\})$ is admissible, and satisfies the hypotheses of Lemma ~\ref{lem:raise-stacky-order}. Observe that $\halfcan'$ itself is an exceptional case, as it has a generator in degree $12 > 2 \cdot 5$. However, after applying Lemma ~\ref{lem:raise-stacky-order}, we see that $P - Q + \frac{3}{7}P_2$ does satisfy the constraints of this theorem, because only relations in degree $\leq 14 = 2 \cdot 7$ are added, and the relation in degree $12$ coming from $R_{\halfcan'}$, lies in a degree less than $2 \cdot 7 = 14$. Therefore, the Theorem holds for $L' = P - Q + \frac{3}{7}P_2$. Then, applying Lemma ~\ref{lem:raise-stacky-order} $\frac{e-7}{2}$ times shows that the Theorem holds for $L = P - Q + \frac{e- 1}{2e}P_1$.

Second, we deal with the case that $\sx$ has at least two stacky points. 
If $(\sx', \Delta, \halfcan')$ is a spin canonical curve so that $L' = P - Q + \frac{1}{3}P_1+ \frac{1}{3}P_2$, then as found in 
Example ~\ref{eg:base-1-33}, the triple $(\sx', 0, P-Q + \frac{1}{3}P_1+\frac{1}{3}P_2)$
satisfies the hypotheses of Lemma ~\ref{lem:raise-stacky-order}.
Therefore, applying Lemma ~\ref{lem:sat-2} $r-2$ times, we see that the theorem holds for $(\sx', \Delta', \halfcan')$ with $L' = P - Q + \sum_{i = 1}^{r}\frac{1}{3}P_i.$ 
Finally, by Corollary ~\ref{cor:raise-stacky-order}, this theorem holds 
for $L' = P - Q + \sum_{i = 1}^{r}\subhalf {e_i}P_i$, as desired. 

\vskip.1in
\noindent
{\sf Case 2b: $\halfcan_X = 0$}

This case is analogous to 2a: If there is only one stacky point, we start at $L = \frac{3}{7}P_1,$ and inductively increment the stabilizer order. Note that by Table ~\ref{table:g-1-base}, $L = \frac{3}{7}P_1,$ will have a relation in degree $15.$ However, once $e_1 \geq 9$, we have $2 \cdot e_1 > 15$, so the theorem holds for such stacky curves. Once the log spin canonical divisor has at least two stacky points, the argument proceeds as in Case 2a.
\end{proof}

\begin{rem}
\label{rem:genus-1-explicit-bound}
In addition to the bound on the degree of the generators and relations, as detailed in Theorem ~\ref{thm:g-1-main}, the proof of Theorem ~\ref{thm:g-1-main} yields an explicit procedure for computing those minimal generators and relations. One can start with the generators and relations found in the base cases and inductively add generators and relations as one adds stacky points and increments stabilizer orders. As described in Remark ~\ref{rem:genus-2-explicit-bound}, when $e := \max(e_1,\ldots, e_r) \geq 7$, there is necessarily a generator in degree $e$ and a relation in degree $2e$.
\end{rem}

%%%%%%%%%%%%%%%%%%%%%%%%%%%% Genus zero %%%%%%%%%%%%%%%%%%%%%%%%%%%%%%%

\section{Genus Zero}
\label{sec:g-0}
We will prove that if $(\sx, \Delta, \halfcan)$ is a log spin
curve and $\sx$ has signature $\sigma := (0; e_1, \ldots, e_r;
\delta)$, then $R(\sx, \Delta, L)$ is generated in degree at most
$e := \max(5, e_1, \ldots, e_r)$ with relations generated in degree
at most $2e$, so long as $\sigma$ does not lie in the finite list
given in Table ~\ref{table:g-0-exceptional}.

As noted in Remark ~\ref{rem:delta-not-1}, $\delta$ is even. Thus,
we can reduce the problem into two cases: $\delta \geq 2$ and
$\delta = 0$. In the former case, $\halfcan$ is linearly equivalent
to an effective divisor, so the result of Theorem \ref{thm:main}
follows immediately by repeatedly applying Lemma ~\ref{lem:sat-1}
to add the necessary stacky points. On the other hand, the proof
when $\delta = 0$ is more involved. We dedicate the remainder of
this section to that case in the following steps: characterizing
saturations (Subsection ~\ref{ssec:g-0-saturation}),
describing base cases (Subsection ~\ref{ssec:g-0-base}), and
presenting exceptional cases (Subsection ~\ref{ssec:g-0-exceptional})
\footnote{Several computations used to generate the tables in
Subsection ~\ref{ssec:g-0-base} and Subsection
~\ref{ssec:g-0-exceptional} were done using a modified version of
the MAGMA code given in the work of O'Dorney \cite{dorney:canonical}.}
Finally, we apply inductive processes using Lemma ~\ref{lem:sat-3}
and Lemma ~\ref{lem:raise-stacky-order} to prove the main theorem
in the full genus zero case (Subsection ~\ref{ssec:g-0-main}).
 
\begin{rem}
Since all points are linearly equivalent on $\BP^1_\Bk$, 
$\halfcan_X \sim n \infty$ for some $n \in \BN$ and
$K_X \sim -2 \infty$. We will use this convention
throughout this section.
\end{rem}

\ssec{Saturation}
\label{ssec:g-0-saturation}
\begin{table}
\begin{tabular}
	{| c | c || c |}
	\hline
	Signature $\sigma$ & Condition & Saturation \\
	\hline
	\hline

	$(0; 3, 3, 3; 0)$ & & $\infty$ \\	\hline

	$(0; 3, 3, 5; 0)$ & & $18$ \\	\hline
	
	$(0; 3, 3, 7; 0)$ & & $12$ \\	\hline
	
	$(0; 3, 3, 9; 0)$ & & $12$ \\	\hline
	
	$(0; 3, 5, 5; 0)$ & & $8$ \\	\hline
	
	$(0; 5, 5, 5; 0)$ & & $8$ \\	\hline
	
	$(0; 3, 3, 3, 3; 0)$ & & $6$ \\	\hline
	
	\hline
	\hline
	
	$(0; 3, 3, \ell; 0)$ & $\ell > 9$ & $9$ \\	\hline
	
	$(0; a, b, c; 0)$ & not listed above & $5$ \\	\hline
	
	$(0; e_1, \ldots, e_r; 0)$ & not listed above & $3$ \\	\hline
\end{tabular}

\caption{Genus 0 Saturation}
\label{table:g-0-sat}
\end{table}

First, we present the saturations of the log spin canonical divisor
(recall Definition ~\ref{defn:sat}) for all cases where $g = 0$ and
$\delta = 0$ in Table ~\ref{table:g-0-sat}. The saturations can
be computed using Riemann--Roch. By classifying the saturations of
all signatures, we can determine the base cases on which we can
apply inductive lemmas from Section ~\ref{sec:induction}. Note that
the saturations of log spin canonical divisors only depend on the
signature here. In Table ~\ref{table:g-0-sat}, exceptional cases
are listed first and generic cases follow.

\ssec{Base Cases}
\label{ssec:g-0-base}
In order to apply Lemma ~\ref{lem:sat-3} and Lemma
~\ref{lem:raise-stacky-order} when $\delta = 0$,
we need to determine appropriate base cases that will
cover all but finitely many signatures by induction.
Here we provide such base cases and demonstrate that
they satisfy all of the necessary conditions of
Lemma ~\ref{lem:sat-3} and Lemma ~\ref{lem:raise-stacky-order}
(e.g. admissibility as defined in Definition \ref{defn:admissible}).
We also show that that the associated log spin
canonical rings are generated in degree at most $e := \max(5, e_1,
\ldots, e_r)$ with relations generated in degree at most $2e$.

\begin{lem}
\label{lem:g-0-admissible-cases}
Let $(\sx', \Delta, \halfcan')$ be a log spin curve with signature
$\sigma := (0; e_1, \ldots, e_r; 0)$. Then, $R' := R_{\halfcan'}$ is
generated by elements of degree at most $e = \max(5 , e_1, \ldots,
e_r)$ with relations in degree at most $2e$. Furthermore, each of
the cases in Table ~\ref{table:g-0-base-cases} satisfy the
conditions of Lemma ~\ref{lem:raise-stacky-order} (i.e.~ either $(\sx',
\halfcan', J) = (\sx', -\infty + \sum_{i = 1}^{r} \subhalf{e_i}
P_i, J)$ is admissible or the stabilizer orders are all $3$ as per
Case ~\ref{custom:three-cases-1} of Remark ~\ref{rem:three-cases}):
 
\rm{
\begin{table}
\begin{tabular}
	{| c || c | c | c |}
	\hline
	Case & Signature $\sigma$ & $J$ & $e$\\
	\hline
	\hline

	(a) & $(0; 3, 3, 11; 0)$ & $\{3\}$ & $11$ \\	\hline

	(b) & $(0; 3, 5, 9; 0)$ & $\{3\}$	& $9$ \\ \hline

	(c) & $(0; 3, 7, 7; 0)$ & $\{2, 3\}$ & $7$ \\ \hline

	(d) & $(0; 5, 5, 7; 0)$ & $\{3\}$	& $7$ \\ \hline
	
	(e) & $(0; 5, 7, 7; 0)$ & $\{2, 3\}$ & $7$ \\ \hline
	
	(f) & $(0; 7, 7, 7; 0)$ & $\{1, 2, 3\}$	& $7$ \\ \hline

	(g) & $(0; 3, 3, 3, 5; 0)$ & $\{4\}$ & $5$ \\ \hline
	
	(h) & $(0; 3, 3, 5, 5; 0)$ & $\{3, 4\}$ & $5$ \\ \hline
	
	(i) & $(0; 3, 5, 5, 5; 0)$ & $\{2, 3, 4\}$ & $5$ \\ \hline
	
	(j) & $(0; 5, 5, 5, 5; 0)$ & $\{1, 2, 3, 4\}$ & $5$ \\ \hline

	(k) &	$(0; 3, 3, 3, 3, 3; 0)$ & $\{1, 2, 3, 4, 5\}$ & $5$ \\ \hline
\end{tabular}

\caption{Genus 0 Base Cases}
\label{table:g-0-base-cases}
\end{table}
}

\end{lem}

\begin{proof}
Recall that the generator and relation degree bounds for case (b)
are proven in Example ~\ref{eg:base-0-377} and the admissibility
condition is checked in Example ~\ref{eg:base-0-377-adm}. For
the remaining cases, we follow a similar method to find a
presentation satisfying the desired conditions. The log spin
canonical ring $R(\sx', 0, L')$ is generated as a $\Bk$-algebra by
elements of degree at most $e$ with relations in degree at most $2e$
for each case as described in the Table
~\ref{table:g-0-base-cases-degrees}.
\begin{table}
\begin{tabular}
	{| c || c | c | c |}
	\hline
	Case & Generator Degrees & Degrees of Relations & $e$\\
	\hline
	\hline

	(a) & $\{3, 7, 9, 11\}$ & $\{14, 18\}$ & $11$ \\	\hline

	(b) & $\{3, 5, 7, 9\}$ & $\{12, 14\}$ & $9$ \\ \hline

	(c) & $\{3, 5, 7, 7\}$ & $\{10, 14\}$ & $7$ \\ \hline
	
	(d) & $\{3, 5, 5, 7\}$ & $\{10, 12\}$	& $7$ \\ \hline
	
	(e) & $\{3, 5, 5, 7, 7\}$ & $\{10, 10, 12, 12, 14\}$	& $7$ \\ \hline
	
	(f) & $\{3, 5, 5, 7, 7, 7\}$ & $\{10, 10, 10, 12, 12, 12, 14, 14, 14\}$	& $7$ \\ \hline

	(g) & $\{3, 3, 4, 5\}$ & $\{8, 9\}$ & $5$ \\ \hline
	
	(h) & $\{3, 3, 4, 5, 5\}$ & $\{8, 8, 9, 9, 10\}$ & $5$ \\ \hline
	
	(i) & $\{3, 3, 4, 5, 5, 5\}$ &
	$\{8, 8, 8, 9, 9, 9, 10, 10, 10\}$ & $5$ \\ \hline
	
	(j) & $\{3, 3, 4, 5, 5, 5, 5\}$ &
	$\{8, 8, 8, 8, 9, 9, 9, 9, 10, 10, 10, 10, 10, 10\}$ & $5$ \\ \hline

	(k) &	$\{3, 3, 3, 4, 4, 5\}$ & $\{6, 7, 7, 8, 8, 8, 9, 9, 10\}$ & $5$ \\ \hline
\end{tabular}

\caption{Generators and Relations for Genus 0 Base Cases}
\label{table:g-0-base-cases-degrees}
\end{table}

We can also always find a presentation for these cases such that
they satisfy \ref{custom:Ad-i} and \ref{custom:Ad-ii} and that $\initial(I')$ is
generated by products of two monomials. Again, the procedure to
verify these is similar to that in Example
~\ref{eg:base-0-377} and Example ~\ref{eg:base-0-377-adm}.

Furthermore, each case always satisfies \ref{custom:Ad-iii} as demonstrated in table ~\ref{table:g-0-base-cases-admissibility}.
Notice that the $e_i$ and $\{e_j' : j \neq i\}$ are equivalent
for any choice of $i \in J$ for these cases, so $\deg \lfloor e_i \halfcan
\rfloor$ and $\max_{k \geq 0} \#S_{(\sigma, J)}(i)$ are
independent of the choice of $i$.

\begin{table}
\begin{tabular}
	{| c | c | c || c | c |}
	\hline
	Case & Signature $\sigma$ & $J$ & $\deg \lfloor e_i L \rfloor$ &
	$\max_{k \geq 0} \#S_{(\sigma, J)}(i)$ \\
	\hline
	\hline

	(a) & $(0; 3, 3, 11; 0)$ & $\{3\}$ & $1$ & $0$ \\	\hline
	
	(b) & $(0; 3, 5, 9; 0)$ & $\{3\}$ & $1$ & $0$ \\ \hline
	
	(c) & $(0; 3, 7, 7; 0)$ & $\{2, 3\}$ & $2$ & $0$ \\ \hline
	
	(d) & $(0; 5, 5, 7; 0)$ & $\{3\}$ & $1$ & $0$ \\ \hline
	
	(e) & $(0; 5, 7, 7; 0)$ & $\{2, 3\}$ & $2$ & $0$ \\ \hline

	(f) & $(0; 7, 7, 7; 0)$ & $\{1, 2, 3\}$ & $3$ & $0$ \\ \hline

	(g) & $(0; 3, 3, 3, 5; 0)$ & $\{4\}$ & $2$ & $0$ \\ \hline
	
	(h) & $(0; 3, 3, 5, 5; 0)$ & $\{3, 4\}$ & $3$ & $0$ \\ \hline
	
	(i) & $(0; 3, 5, 5, 5; 0)$ & $\{2, 3, 4\}$ & $4$ & $0$ \\ \hline
	
	(j) & $(0; 5, 5, 5, 5; 0)$ & $\{1, 2, 3, 4\}$ & $5$ & $0$ \\ \hline

	(k) & $(0; 3, 3, 3, 3, 3; 0)$ & $\{1, 2, 3, 4, 5\}$ & $5$ & $0$ \\ \hline
\end{tabular}	

\caption{Checking \ref{custom:Ad-iii} for Genus 0 Base Cases}
\label{table:g-0-base-cases-admissibility}
\end{table}

Thus, all of the cases are admissible and satisfy the additional
desired conditions.
\end{proof}

\ssec{Exceptional Cases}
\label{ssec:g-0-exceptional}
In this subsection, we describe the cases that are not covered by induction, which are also the only exceptions to Theorem ~\ref{thm:main} in the case $g = 0$.
In table ~\ref{table:g-0-exceptional} We present the explicit generators and relations for the remaining
cases given by signatures in the finite set

\begin{align*}
	S &:= \{(0; 3, 3, \ell; 0) : 3 \leq \ell \leq 9 \text{ odd}\} \\
		&\cup \{(0; 3, 5, 5; 0) ,(0; 3, 5, 7; 0), (0; 5, 5, 5; 0), (0; 3, 3, 3, 3; 0)\}
\end{align*}

\begin{table}
\begin{tabular}
	{| c || c | c | c |}
	\hline
	Signature $\sigma$ & Generator Degrees & Degrees of Relations & $e$ \\
	\hline
	\hline

	$(0; 3, 3, 3; 0)$ & $\{3\}$ & $\emptyset$ & $5$ \\	\hline

	$(0; 3, 3, 5; 0)$ & $\{3, 10, 15\}$ & $\{30\}$ & $5$ \\	\hline
	
	$(0; 3, 3, 7; 0)$ & $\{3, 7, 12\}$ & $\{24\}$ & $7$ \\	\hline
	
	$(0; 3, 3, 9; 0)$ & $\{3, 7, 9\}$ & $\{21\}$ & $9$ \\	\hline
	
	$(0; 3, 5, 5; 0)$ & $\{3, 5, 10\}$ & $\{20\}$ & $5$ \\	\hline
	
	$(0; 3, 5, 7; 0)$ & $\{3, 5, 7\}$ & $\{17\}$ & $7$ \\	\hline
	
	$(0; 5, 5, 5; 0)$ & $\{3, 5, 5\}$ & $\{15\}$ & $5$ \\	\hline
	
	$(0; 3, 3, 3, 3; 0)$ & $\{3, 3, 4\}$ & $\{12\}$ & $5$ \\	\hline
\end{tabular}

\caption{Genus 0 Exceptional Cases}
\label{table:g-0-exceptional}
\end{table}

\begin{rem}
These cases give all of the exceptions to the $e$ and $2e$ bounds
on the generator and relation degree. Notice that each of these
exceptional cases, apart from $(0; 3, 5, 7; 0)$, also has
exceptional saturation as seen in Table ~\ref{table:g-0-sat}.
Intuitively, these exceptional saturations can be viewed as
``forcing'' generators and relations in higher degrees than expected.
\end{rem}

\ssec{Main Theorem for Genus Zero}
\label{ssec:g-0-main}
Now we can combine the base cases from Subsection \ref{ssec:g-0-base} with the inductive lemmas of Section \ref{sec:induction}.

\begin{thm}
\label{thm:g-0-main}
Let $(\sx, \Delta, \halfcan)$ log spin curve with signature
$\sigma := (0; e_1, \ldots, e_r; \delta)$.
Then, the log spin canonical ring $R(\sx, \Delta, \halfcan)$
is generated as a $\Bk$-algebra by elements of degree at most $e =
\max(5, e_1, \ldots, e_r)$ and has relations in degree at most $2e$,
so long as $\sigma$ does not lie in the finite list of exceptional
cases in Table ~\ref{table:g-0-exceptional}.
\end{thm}

\sssec*{Idea of Proof:}
The method of this proof is almost identical to that of Theorem
~\ref{thm:g-1-main}. When $\delta > 0$, we first add in log points,
and then increment the stabilizer orders of stacky points, checking
that the theorem holds at each step. The more technical case occurs
when $\delta = 0$. In this case, we increment the stabilizer
orders of stacky points starting from one of the base cases, and
check that every stacky curve can be reached by a sequence of
admissible incrementations from a base case.

\begin{proof}
If $\sx$ has no stacky points, then we can assume that $\halfcan \sim
n \cdot \infty$ with $n \in \BZ_{\geq -1}$. This is a classical
case done by Voight and Zureick-Brown \cite[Section 4.2]
{vzb:stacky}. When $n = -1$, then $R_\halfcan = \Bk$.  When $n = 0$,
then $R_\halfcan = \Bk[x]$.  When $n > 0$, inductively applying Lemma
\ref{lem:sat-1} tells us that $R_\halfcan$ is generated in degree 1
with relations generated in degree 2.

First, let us consider the case when $\delta \geq 2$. In any such
case, $\lfloor \halfcan \rfloor$ is an effective divisor and the
conditions of Lemma ~\ref{lem:sat-1} are satisfied. Thus, we can
apply the Lemma ~\ref{lem:sat-1} inductively from the classical
case with no stacky points to get that $R(\sx, \Delta, \halfcan)$
is generated up to degree $e$.

By Remark ~\ref{rem:delta-not-1}, it only remains to deal with the
case $\delta = 0$, so $\halfcan$ is not necessarily effective. Let
the signature $\sigma$ be such that it is not one of the
exceptional cases contained in Table ~\ref{table:g-0-exceptional}. 
We get the following three cases, depending on the value of $r$:

\vskip.1in
\noindent
{\sf Case 1: $r < 3$}

If $r < 3$, then $\deg 
\lfloor k \halfcan \rfloor < 0$ for all $k \geq 0$ so we have the
trivial case where $R(\sx, \Delta, \halfcan) = \Bk$.

\vskip.1in
\noindent
{\sf Case 2: $3 \le r \le 5$}

If $3 \le r \leq 5$ and $\sigma$ is not one of the exceptional cases,
then we may apply Lemma ~\ref{lem:g-0-admissible-cases} and
Corollary ~\ref{cor:raise-stacky-order} to an appropriate base case
from Table ~\ref{table:g-0-base-cases} and deduce that $R(\sx,
\delta, \halfcan)$ is generated up  to degree $e := \max(5, e_1,
\ldots, e_r)$ with relations generated up to degree $2e$.

\vskip.1in
\noindent
{\sf Case 3: $r > 5$}

If $r > 5$, then we can use Lemma
~\ref{lem:sat-3} to add stacky points with
stabilizer order $3$ to case (k) of Table
~\ref{table:g-0-base-cases}, which corresponds to $(\sigma
= (0; 3, 3, 3, 3, 3; 0), J = \{1, 2, 3, 4, 5\})$. This case
satisfies the conditions of Lemma
~\ref{lem:sat-3} (recall from Table ~\ref
{table:g-0-sat} that $\sat(\Eff(\sigma)) = 3$), and the
immediate consequence of parts (a) and (c) of Lemma
~\ref{lem:sat-3} is that any $R(\sx', \Delta,
\halfcan')$ corresponding to signatures $\sigma'$ with ramification
orders all equal to $3$ for any $r > 5$ is generated up to
degree $e' := \max(5, e_1', \ldots, e_r')$ with relations generated
up to degree $2e'$. Furthermore, these cases satisfy all of the
conditions of Lemma ~\ref{lem:raise-stacky-order}. Now we can
apply Corollary ~\ref{cor:raise-stacky-order} to deduce that
$R(\sx, \delta, \halfcan)$ is generated up to degree $e :=
\max(5, e_1, \ldots, e_r)$ with relations generated up to degree
$2e$.
\end{proof}

\begin{rem}
\label{rem:genus-0-explicit-bound}
The proof of Theorem ~\ref{thm:g-0-main} in genus zero gives an
explicit construction of the generators and relations for the log
spin canonical ring $R_\halfcan$. This is similar to the case of genus 1
in Remark ~\ref{rem:genus-1-explicit-bound}. Furthermore, there is
a generator in degree $e$ and a relation in degree at least $2e-4$
when $e := \max(e_1, \ldots, e_r)$ is at least $7$ (see Remark ~\ref
{rem:genus-2-explicit-bound}). This can be seen from the inductive
application of Lemmas ~\ref{lem:sat-1}, \ref{lem:sat-2}, and \ref
{lem:sat-3}.
\end{rem}

\begin{rem}
\label{rem:cusp-generation-4}
Here, we describe how to obtain a slightly better bound for our
application to modular forms from Example ~\ref{eg:congruence-bounds} in the cases $g=0$ and $g=1$.
When $g=0$, a careful scrutiny of Theorem ~\ref{thm:g-0-main} reveals that,
if $\Delta > 0$ and $\sx$ has signature $(0; 3, \ldots, 3; \delta)$,
then $R_\halfcan$ is generated in weight at most 4. Since $\delta > 0$,
$\halfcan_X$ is effective. Additionally, $R_{\halfcan_X}$ is generated in weight 1,
and inductive applications of Lemma ~\ref{lem:sat-1} only add
generators in weights 3 and 4 and relations in weight at most 8.
Therefore, $R_\halfcan$ is generated in weight at most 4 with relations in
weight at most 8. Note that a similar analysis of the proof of
Theorem ~\ref{thm:g-1-main} yields that when $g = 1$, congruence
subgroups are generated in weight at most 4 with relations in
weight at most 8.
\end{rem}

%%%%%%%%%%%%%%%%%%%%%%%%% Further Research %%%%%%%%%%%%%%%%%%%%%%%%%%%%

\section{Further Research}
\label{sec:further-questions}
In this section, we present several directions for further research.
\begin{enumerate}
	\item As noted in Remark ~\ref{rem:explicit-generators}, the 
		proof of Theorem ~\ref{thm:main} 
		gives an 
		explicit procedure for computing the generators and relations of $R_
		\halfcan$ when the genus of $\sx$ is $0$ or $1$.  When $\sx$ has genus at least 2, Lemmas ~\ref{lem:sat-1} and $\ref{lem:sat-2}$ allow us to explicitly construct a presentation of $R_\halfcan$ from a presentation of $R_{\halfcan_X}$ where $X$ is the coarse space $X$.  However, obtaining a presentation for $X$ requires nontrivial computation.  This suggests the following Petri-like question:
		\begin{question}
		\label{ques:spin-canonical-petri}
			Is there a general structure theorem describing a set of minimal 
			generators and relations of $R_{\halfcan}$ where $(X, \Delta, L)$ is a log spin curve with no stacky points?
		\end{question}
		
\item One direction for further research is to extend the results
	of this paper to divisors $D \in \di \sx$ on a stacky curve $\sx$, 
	where $nD \sim K$ for some integer $n$ greater than $2$. The
	canonical rings	of such divisors often arise as rings of
	fractional weight $\frac{2}{n}$ modular forms.
	For more details on fractional weight modular forms, see Adler and
	Ramanan \cite[p. 96]{adler:moduli} and Milnor
	\cite[$\mathsection$ 6]{milnor:fractional-weight}.

\begin{question}
		\label{ques:fractional-weight}
			If $\sx$ is a stacky curve and $D \in \di \sx$ with $nD \sim K$, 
			where $K$ is the canonical divisor of $\sx$, can one bound the 
			degrees of generators and relations of $R_D$?
		\end{question}
		When $g=0$ and $D$ is effective, inductively applying Lemma ~\ref{lem:sat-1}  
		gives an affirmative answer to this question: If $\sx$ has signature 
		$(g; e_1, \ldots, e_r; \delta)$ then $R_D$ is generated in degree at 
		most $e = \max(e_1, \ldots, e_r)$ with relations in degree at most
		$2e$. 
		It may be possible to modify the proof of Lemma
		~\ref{lem:raise-stacky-order} to extend to the setting of fractional 
		weight modular forms. 
		Suitable generalizations of the 
		lemmas of Section ~\ref{sec:induction} might allow one to 
		follow similar questions to this paper and provide a general 
		answer to Question ~\ref{ques:fractional-weight}.
		\\

	\item The generic initial ideal
		encapsulates the idea of whether the relations for $R_\halfcan$ are 
		generically chosen. See Voight and Zureick-Brown \cite[Definition 2.2.7]{vzb:stacky} for a precise definition of the generic initial ideal.
		The proof of Theorem
		~\ref{thm:main} is decidedly non-generic. 
		In particular, Lemma ~\ref{lem:raise-stacky-order} 
		constructs generators with non-maximal pole orders at certain 
		points, making the relations non-generic. 
		\begin{question}
		\label{ques:generic-initial}
			If $(\sx, \Delta, \halfcan)$ is a log spin curve, can one write down the
			generic initial ideal explicitly?
		\end{question}
		
	\item In Subsection \ref{ssec:bounds-high-genus}, we reference 
		the work of Reid \cite[Theorem 3.4]{reid:infinitesimal}. We use his 
		proof that the 
		spin canonical ring is generated in degree at most 5 with
		relations in degree at most 10 in the 			
		non-log, non-stacky case when genus is at least 2. We extend this 
		bound of 5 and 10 to the log case, and then apply our inductive 
		lemmas to add stacky points and obtain bounds of $e = \max(5,e_1, \ldots, e_r)$ and $2e$.  However, Reid in fact proves something slightly stronger \cite[Theorem 3.4]{reid:infinitesimal}: that in 
		most cases his bound is actually 3 and 6 with well-characterized 
		exceptions.
		Generalizing this slightly stronger bound to the (non-stacky) log case case would allow us to inductively apply the lemmas from Section ~\ref{sec:induction} and improve Theorem ~\ref{thm:main} as follows:
		\begin{question}\label{ques:reduce-5,10-to-4,8}
			When $g \geq 2,$ can the bounds in Theorem \ref{thm:main} on 
			the degrees of generation and relations be reduced from 
			$e := \max(5, e_1, \ldots, e_r)$ and $\max(10, 2e_1, \ldots, 2e_r)$ to 
			$e' := \max(4, e_1, \ldots, e_r)$ and $2e'$, 
			apart from well a characterized list of families?
		\end{question}
		\begin{rem}\label{rem:reduce-3,6}
			Note that when $L$ is not effective and $\sx$ has a stacky point, $R_\halfcan$ must have a generator in degree 4 with maximal pole order at one of the stacky points.  Therefore, these bounds cannot in general be reduced further to 
			$e'' := \max(3, e_1, \ldots, e_r)$ and $2e''$.
			\end{rem}
	\item While Theorem ~\ref{thm:main} gives a set of 
		generators and relations
		for the log spin canonical ring $R_\halfcan,$ these sets are not necessarily minimal.
		In many of the $g = 0$ and $g = 1$ cases, 
		it is not too difficult to see that our inductive procedure
		yields a minimal set of relations for $R_\halfcan$. One might investigate 
		whether the generators and relations given by the inductive proof 
		of Theorem ~\ref{thm:main} are always minimal.
		\end{enumerate}

%%%%%%%%%%%%%%%%%%%%%%%%% Acknowledgements %%%%%%%%%%%%%%%%%%%%%%%%%%%%

\section{Acknowledgments}
We are grateful to David Zureick-Brown for introducing us to the
study of stacky canonical rings, for providing invaluable guidance,
and for his mentorship. We also thank Ken Ono and the Emory
University Number Theory REU for arranging our project and
providing a great environment for mathematical learning and
collaboration. We thank Brian Conrad, John Voight, and Shou-Wu
Zhang for helpful comments on this paper and thank Jorge Neves
for pointing out the bounds given by Miles Reid's work on the
maximal degrees of minimal generators and relations of spin
canonical rings in high genera.
We thank Miles Reid for clarifying how to
find the Hilbert series of stacky curves.
The first author would also like to
mention Joseph Harris and Evan O'Dorney for several enlightening
conversations related to Petri's Theorem. Finally, we gratefully
acknowledge the support of the National Science Foundation (grant
number DMS-1250467). We deeply appreciate all of the support that
has made our work possible.

%%%%%%%%%%%%%%%%%%%%%%%%%%%% References %%%%%%%%%%%%%%%%%%%%%%%%%%%%%%%

\nocite{*}
\bibliography{bibliography-arXiv-v3}{}
\bibliographystyle{alpha}

\end{document}